\documentclass[leqno,amsthm,secthm]{elsart1p}
\usepackage{amsmath,amssymb}

\newcommand{\C}{{\mathbb{C}}}
\newcommand{\F}{{\mathbb{F}}}
\newcommand{\Q}{{\mathbb{Q}}}
\newcommand{\Z}{{\mathbb{Z}}}

\newcommand{\fC}{{\mathfrak{C}}}
\newcommand{\fS}{{\mathfrak{S}}}
\newcommand{\fp}{{\mathfrak{p}}}
\newcommand{\fq}{{\mathfrak{q}}}
\newcommand{\fs}{{\mathfrak{s}}}
\newcommand{\ft}{{\mathfrak{t}}}

\newcommand{\cO}{{\mathcal{O}}}
\newcommand{\cH}{{\mathcal{H}}}

\newcommand{\ba}{{\mathbf{a}}}
\newcommand{\bc}{{\mathbf{c}}}

\newcommand{\Irr}{{\operatorname{Irr}}}

\renewcommand{\leq}{\leqslant}
\renewcommand{\geq}{\geqslant}
\renewcommand{\atop}[2]{\genfrac{}{}{0pt}{}{#1}{#2}}

\begin{document}
\begin{frontmatter}
\title{James' Conjecture for Hecke algebras of exceptional type, I}
\author{Meinolf Geck}
\address{Department of Mathematical Sciences, King's College, 
Aberdeen University, AB24 3UE, Scotland, UK}
\ead{geck@maths.abdn.ac.uk}
\author{J\"urgen M\"uller}
\address{Lehrstuhl D f\"ur Mathematik, RWTH Aachen, Templergraben 64, 
D-52062 Aachen, Germany}
\ead{Juergen.Mueller@math.rwth-aachen.de}
\dedicated{To Gus Lehrer on his $60$th birthday}

\begin{abstract} In this paper, and a second part to follow, we complete the 
programme (initiated more than $15$ years ago) of determining the 
decomposition numbers and verifying James' Conjecture for 
Iwahori--Hecke algebras of exceptional type. The new ingredients which 
allow us to achieve this aim are:
\begin{itemize}
\item[$\bullet$] the fact, recently proved by the first author, that all 
Hecke algebras of finite type are cellular in the sense of Graham--Lehrer,
and
\item[$\bullet$] the explicit determination of $W$-graphs for the irreducible
(generic) representations of Hecke algebras of type $E_7$ and $E_8$ by 
Howlett and Yin.
\end{itemize}
Thus, we can reduce the problem of computing decomposition numbers 
to a manageable size where standard techniques, e.g., Parker's 
{\sf MeatAxe} and its variations, can be applied. In this part, we
describe the theoretical foundations for this procedure.
\end{abstract}

\begin{keyword} 
Hecke algebra \sep decomposition numbers \sep James' conjecture 
\MSC Primary 20C08
\end{keyword}
\end{frontmatter}

\section{Introduction} \label{sec1}
Let $k$ be a field and $q$ a non-zero element of $k$. Let $H_n(k,q)$ be 
the Iwahori--Hecke algebra of type $A_{n-1}$ with parameter $q$; this is 
a certain deformation of the group algebra of the symmetric group $\fS_n$.
In order to study the representation theory of $H_n(k,q)$, Dipper and James 
\cite{DJ} developed a $q$-version of the classical theory of Specht modules 
for $\fS_n$.  In this framework, one obtains a natural parametrization of 
$\Irr(H_n(k,q))$ (the set of irreducible representations, up to isomorphism)
in terms of $e$-regular partitions, where  the parameter $e$ is defined by
\[ e=\min\{i\geq 2\mid 1+q+q^2+\cdots + q^{i-1}=0\}.\]
(We set $e=\infty$ if no such $i$ exists.) If $k$ has characteristic~$0$,
then we also know how to determine the dimensions of the irreducible 
representations, thanks to the Lascoux--Leclerc--Thibon conjecture \cite{llt}
and its proof by Ariki \cite{Ar}. However, the analogous problem for
$k$ of positive characteristic is completely open. 

Assume now that $e<\infty$ and $\mbox{char}(k)=\ell>0$. Based on empirical 
evidence for $n=2,3,\ldots,10$, James \cite{Ja} made the remarkable
conjecture that if $e\ell>n$, then $\Irr(H_n(k,q))$ only depends on~$e$. 
More precisely, James predicts that $\Irr(H_n(k,q))$ could be obtained from 
the $\C$-algebra $H_n(\C,\sqrt[e]{1})$ by a process of $\ell$-modular 
reduction. Shortly afterwards, the first-named author \cite{mybaum} 
formulated a version of James' conjecture for Iwahori--Hecke algebras 
associated to finite Weyl groups in general, and proved that it holds in 
the so-called ``defect $1$ case''. (In type $A_{n-1}$, this corresponds to 
the case where $e$ divides exactly one of the numbers $2,3,\ldots,n$.) The 
article \cite{mybaum} also contains an argument which shows that the 
irreducible representations of any Iwahori--Hecke algebra over a field of 
characteristic $\ell>0$ can always be obtained by $\ell$-modular reduction 
from an algebra in characteristic $0$, as long as $\ell$ is large enough. 
Thus, James' conjecture and its generalizations are really about finding
the correct bound for~$\ell$.

By ad hoc computational methods, the general version of James' conjecture 
has been shown to hold for Iwahori--Hecke algebras of type $F_4$ and $E_6$; 
see \cite{GeLu}, \cite{mye6}. These methods, however, turned out to be 
completely inadequate to deal with algebras of larger rank; in particular,
types $E_7$ and $E_8$ remained far out of reach. 

Using the Kazhdan--Lusztig theory of cells \cite{Lusztig03} and the 
Graham--Lehrer concept of abstract ``cell data'' \cite{GrLe}, it was 
recently shown in \cite{mycell} that a suitable theory of ``Specht modules'' 
exists for Iwahori--Hecke algebras associated to finite Weyl groups in 
general. First of all, this has the theoretical implication that we can now
formulate a general version of James' conjecture which is, perhaps, more 
natural than the one in \cite{mybaum}. Furthermore, this has the practical 
implication of leading to an algorithm for verifying the general version 
of James' conjecture, in which the main issue is the determination of the 
invariant bilinear form (and its rank) on a ``cell representation''. 

In order to make this work, a number of problems have to be resolved. To 
begin with, we need explicit models for those ``cell representations''.
For $W$ of exceptional type, we will see that such models are given by the 
$W$-graph representations which were recently obtained by Howlett and Yin 
\cite{How}, \cite{yin0} and which are readily accessible through Michel's 
development version \cite{jmich} of the computer algebra system 
{\sf CHEVIE} \cite{chv}. Then the determination of the invariant bilinear 
form essentially amounts to solving a system of linear equations. This 
works fine for dimensions of up to around $2500$, but some more refined 
methods are necessary for dealing with the large representations (of 
dimension up to $7168$) in type $E_8$. The discussion of these finer 
computational methods is beyond the scope of the present article and can 
be found in \cite{gemu2}.

Still, with all these new tools at hand, the computations required to 
determine the Gram matrices of the invariant bilinear forms for large
representations in type $E_8$ takes {\it several months} of CPU time on 
modern computers. Note, however, that once these matrices have been computed,
it is relatively easy to verify that they indeed define invariant bilinear 
forms and to compute their ranks for various specialisations. It is planned 
to create a data base which makes these data generally available. 

This paper is organised as follows. In Section~2, we recall the construction
of ``cell data'' \`a la Graham--Lehrer in Iwahori--Hecke algebras associated 
to finite Weyl groups. We also discuss the example of type $G_2$, which 
provides a first illustration for the phenomenon expressed in James' 
conjecture. In Section~3, we formulate the general version of James' 
conjecture using the new approach based on cell representations. The
equivalent formulation in Corollary~\ref{Jconj1} provides the conceptual
basis for the algorithm for verifying James' conjecture. In Section~4, we 
discuss the main computational issues in this algorithm and show how they 
can be solved---at least in principle. In particular, in \S \ref{subP2}, we
prove a general result which allows us to verify that the Howlett--Yin
$W$-graph representations do provide suitable models for the ``cell 
representations''. This fact raises a general question about $W$-graph
representations which is formulated as Conjecture~\ref{Cwgraph}. 

\section{Cellular bases and cell representations} \label{sec2}
Let $W$ be an irreducible finite Weyl group with generating set $S$. 
Let $R \subseteq \C$ be a subring and $A=R[v,v^{-1}]$ the ring of Laurent 
polynomials in an indeterminate~$v$. Let $\cH$ be the corresponding 
$1$-parameter Iwahori--Hecke algebra over $A$. As an $A$-module, $\cH$ is 
free with basis $\{T_w\mid w\in W\}$; the multiplication is given by
\[ T_sT_w=\left\{\begin{array}{cl} T_{sw} &\quad\mbox{if $l(sw)=l(w)+1$},\\
uT_{sw}+(u-1)T_w & \quad \mbox{if $l(sw)=l(w)-1$},\end{array}
\right.\]
where $u=v^2$, $s\in S$ and $w\in W$. Here, $l(w)$ denotes the length of 
$w\in W$. For the general theory of Iwahori--Hecke algebras, we refer to 
\cite{ourbuch}. These algebras, and their specialisations, play an important 
role in the representation theory of finite reductive groups; see, for 
example, \cite[Chap.~0]{Lusztig03}, \cite{laus05}.

In order to specify a {\em cell datum} for $\cH$ in the sense of Graham and 
Lehrer \cite[Def.~1.1]{GrLe}, we must specify a quadruple $(\Lambda,M,C,*)$
satisfying the following conditions.
\begin{description}
\item[(C1)] $\Lambda$ is a partially ordered set, $\{M(\lambda) \mid
\lambda \in \Lambda\}$ is a collection of finite sets  and
\[ C \colon \coprod_{\lambda \in \Lambda} M(\lambda)\times M(\lambda)
\rightarrow \cH \]
is an injective map whose image is an $A$-basis of $\cH$;
\item[(C2)] If $\lambda \in \Lambda$ and $\fs,\ft\in M(\lambda)$, write
$C(\fs,\ft)=C_{\fs,\ft}^\lambda \in \cH$. Then $* \colon \cH \rightarrow
\cH$ is an $A$-linear anti-involution such that $(C_{\fs,\ft}^\lambda)^*=
C_{\ft,\fs}^\lambda$.
\item[(C3)] If $\lambda \in \Lambda$ and $\fs,\ft\in M(\lambda)$, then for
any element $h \in \cH$ we have
\[ hC_{\fs,\ft}^\lambda\equiv \sum_{\fs'\in M(\lambda)} r_h(\fs',\fs)\,
C_{\fs',\ft}^\lambda\quad \bmod \cH(<\lambda),\]
where $r_h(\fs',\fs) \in A$ is independent of $\ft$ and where $\cH(<\lambda)$
is the $A$-submodule of $\cH$ generated by $\{C_{\fs'',\ft''}^\mu
\mid \mu <\lambda;\; \fs'',\ft''\in M(\mu)\}$.
\end{description}
For this purpose, we first need to recall some basic facts about the
representations of $W$ and $\cH_K=K\otimes_A \cH$, where $K$ is the
field of fractions of~$A$.

It is known that $\Q$ is a splitting field for $W$; see, for 
example, \cite[6.3.8]{ourbuch}. We will write
\[ \Irr(W)=\{E^\lambda \mid \lambda \in \Lambda\}, \qquad  d_\lambda=
\dim E^\lambda,\]
for the set of irreducible representations of $W$ (up to equivalence),
where $\Lambda$ is some finite indexing set. Now, the algebra $\cH_K$ is 
known to be split semisimple; see \cite[9.3.5]{ourbuch}. Furthermore, by 
Tits' Deformation Theorem, the irreducible representations of $\cH_K$ (up 
to isomorphism) are in bijection with the irreducible representations 
of $W$; see \cite[8.1.7]{ourbuch}. Thus, we can write
\[ \Irr(\cH_K)=\{E^\lambda_v \mid \lambda \in \Lambda\}.\]
The correspondence $E^\lambda \leftrightarrow E^\lambda_v$ is uniquely
determined by the following condition:
\[ \mbox{trace}\bigl(w,E^\lambda\bigr)=\mbox{trace}\bigl(
T_w,E^\lambda_v\bigr)\Big|_{v=1} \qquad \mbox{for all $w \in W$};\]
note that $\mbox{trace}\bigl(T_w,E^\lambda_v\bigr) \in A$ for all
$w\in W$. 

The algebra $\cH$ is {\em symmetric} with respect to the trace form 
$\tau\colon \cH \rightarrow A$ defined by $\tau(T_1)=1$ and $\tau(T_w)=0$
for $1\neq w\in W$. Hence we have the following orthogonality relations 
for the irreducible representations of $\cH_K$:
\[ \sum_{w \in W} u^{-l(w)}\mbox{trace}\bigl(T_w,E^\lambda_v\bigr)
\,\mbox{trace}\bigl(T_{w^{-1}},E_v^\mu\bigr)=\left\{\begin{array}{cl}
d_\lambda\,\bc_\lambda & \quad \mbox{if $\lambda=\mu$},\\ 0 & \quad
\mbox{if $\lambda \neq\mu$},\end{array}\right.\]
where $0 \neq \bc_\lambda \in {\Z}[u,u^{-1}]$; see \cite[8.1.7 and 
9.3.6]{ourbuch}. Following Lusztig, we write
\[ \bc_\lambda=f_\lambda\, u^{-\ba_\lambda}+\mbox{combination of strictly
higher powers of $u$},\]
where $\ba_\lambda,f_\lambda$ are integers such that $\ba_\lambda\geq 0$
and $f_\lambda>0$; see \cite[9.4.7]{ourbuch}. These integers are explicitly 
known for all types of $W$; see Lusztig \cite[Chap.~4]{LuBook} or 
\cite[Chap.~22]{Lusztig03}. 

\medskip
\begin{rem} \label{equalp} Since we are in the equal parameter case, the
Laurent polynomials $\bc_\lambda$ have the following properties: Each 
$\bc_\lambda$ divides the Poincar\'e polynomial $P_W=\sum_{w \in W} 
u^{l(w)}$ in ${\Q}[u,u^{-1}]$; furthermore, we have 
\[\bc_\lambda=f_\lambda u^{-\ba_\lambda}\, \tilde{\bc}_\lambda \qquad
\mbox{where $\tilde{\bc}_\lambda \in {\Z}[u]$ is monic and divides
$P_W$}.\]
(For these facts, see \cite[9.3.6]{ourbuch} and the references there.) It 
is well-known (see, for example, \cite[\S 9.4]{Ca1}) that 
\[ P_W=\prod_{1\leq i \leq |S|} \frac{u^{d_i}-1}{u-1}\]
where $d_1,\ldots,d_{|S|}$ are the so-called {\it degrees} of $W$; we 
have $|W|=d_1\cdots d_{|S|}$. By \cite[\S 10.2]{Ca1}, the degrees for 
the various types of $W$ are given as follows: 
\[ \renewcommand{\arraystretch}{1.2} \begin{array}{cc} \hline
\mbox{Type} & \mbox{degrees $d_i$}\\ \hline
A_{n-1} & 2,3,4,\ldots,n \\
B_n,C_n & 2,4,6, \ldots, 2n \\
D_n & 2,4,6,\ldots,2(n-1),n\\ \hline  & \\ & \end{array}\qquad
\begin{array}{cc} \hline
\mbox{Type} & \mbox{degrees $d_i$}\\ \hline
G_2 & 2,6 \\
F_4 & 2,6,8,12 \\
E_6 & 2,5,6,8,9,12 \\
E_7 & 2,6,8,10,12,14,18\\
E_8 & 2,8,12,14,18,20,24,30 \\ \hline
\end{array}\]
\end{rem}

We are now ready to define a ``cell datum'' of $\cH$. The required 
quadruple $(\Lambda,M,C,*)$ is given as follows. Let $\Lambda$ be an 
indexing set for the irreducible representations of $W$, as above. For 
$\lambda\in \Lambda$, we set $M(\lambda)=\{1,\ldots, d_\lambda\}$. Using 
the $\ba$-invariants, we define a partial order $\preceq$ on $\Lambda$ by
\[ \lambda \preceq \mu \qquad \stackrel{\text{def}}{\Leftrightarrow} \qquad
\lambda=\mu \quad \mbox{or}\quad \ba_\lambda >\ba_\mu.\]
Thus, $\Lambda$ is ordered according to {\em decreasing} $\ba$-value. Next,
we define an $A$-linear anti-involution $* \colon \cH \rightarrow \cH$
by $T_w^*=T_{w^{-1}}$ for all $w \in W$. Thus, $T_w^*=T_w^{\,\flat}$ in the
notation of \cite[3.4]{Lusztig03}. 

The trickiest part is, of course, the definition of the basis elements 
$C_{\fs,\ft}^\lambda$ for $\fs,\ft\in M(\lambda)$. Let $\{c_w\mid w\in W\}$
be the Kazdan--Lusztig basis of $\cH$, as  constructed in 
\cite[Theorem~5.2]{Lusztig03}. Given $x,y \in W$, we write $c_xc_y=
\sum_{z\in W} h_{x,y,z}c_z$ where $h_{x,y,z}\in A$. Following Lusztig 
\cite[13.6]{Lusztig03}, we use the structure constants $h_{x,y,z}$ to 
define a function $\ba \colon W \rightarrow \Z_{\geq 0}$ by
\[ \ba(z):=\min \{i \geq 0 \mid v^i h_{x,y,z}\in {\Z}[v] \mbox{ for all
$x,y \in W$}\} \qquad \mbox{for all $z \in W$}.\]
As in [{\em loc.\ cit.}], we usually work with the elements $c_w^\dagger$ 
obtained by applying the unique $A$-algebra involution $\cH \rightarrow 
\cH$, $h \mapsto h^\dagger$ such that $T_s^\dagger=-T_s^{-1}$ for any 
$s\in S$; see \cite[3.5]{Lusztig03}. We can now state:

\medskip
\begin{thm}[Geck \protect{\cite[Theorem 3.1]{mycell}}] \label{Hcellbase} 
Assume that the subring $R \subseteq \C$ is chosen such that all bad primes 
for $W$ are invertible in $R$. Then there is a cell datum $(\Lambda,M,C,*)$
for $\cH$ where $\Lambda$, $M$, $*$ are as specified above and, for each 
$\lambda \in \Lambda$ and $\fs,\ft\in M(\lambda)$, the element $C_{\fs,
\ft}^\lambda$ is a $\Z$-linear combination of basis elements $c_w^\dagger$ 
where $\ba(w)=\ba_\lambda$.
\end{thm}

\medskip
Here, a prime number $p$ is called {\em bad} for $W$ if $p$ divides 
$f_\lambda$ for some $\lambda \in \Lambda$. Otherwise, $p$ is called 
{\em good}. This corresponds to the familiar definition of ``bad'' primes; 
see Lusztig \cite[Chap.~4]{LuBook}. The conditions for being good for the 
various types of $W$ are as follows:
\begin{center} $\begin{array}{rl} A_n: & \mbox{no condition}, \\
B_n, C_n, D_n: & p \neq 2, \\
G_2, F_4, E_6, E_7: &  p \neq 2,3, \\
E_8: & p \neq 2,3,5.  \end{array}$
\end{center}
For the rest of this paper, we shall now make the definite choice where the 
ring $R$ consists of all fractions $a/b\in \Q$ such that $a \in \Z$ and
$0 \neq b\in \Z$ is divisible by bad primes only.

\medskip
\begin{rem} \label{integral} For future reference, we remark 
that, if $h \in \cH$ is a ${\Z}[v,v^{-1}]$-linear combination of basis
elements $\{T_w \mid w \in W\}$, then we also have  
\[ r_h(\fs',\fs) \in {\Z}[v,v^{-1}] \qquad \mbox{for all $\lambda \in
\Lambda$ and $\fs,\fs' \in M(\lambda)$};\]
see the explicit formula for $r_h(\fs',\fs)$ in Step~3 of the proof 
of \cite[Theorem 3.1]{mycell}.
\end{rem}

\medskip
Following Graham and Lehrer \cite{GrLe}, we can perform the following 
constructions. Given $\lambda \in \Lambda$, let 
$W^\lambda$ be a free $A$-module with basis $\{C_\fs\mid \fs \in
M(\lambda)\}$. Then $W^\lambda$ is a left $\cH$-module, where the action
is given by 
\[ h.C_\fs=\sum_{\fs' \in M(\lambda)} r_h(\fs',\fs)\, C_{\fs'}.\] 
Furthermore, we can define a symmetric bilinear form $\phi^\lambda \colon 
W^\lambda \times W^\lambda \rightarrow A$ by 
\[ \phi^\lambda(C_\fs,C_\ft)=r_h(\fs,\fs) \qquad \mbox{where $\fs,\ft\in
M(\lambda)$ and $h=C_{\fs, \ft}^\lambda$}.\] 
We have $\phi^\lambda(T_w.C_\fs,C_\ft)=\phi^\lambda(C_\fs,T_{w^{-1}}.C_\ft)$ 
for all $\fs,\ft \in M(\lambda)$ and $w \in W$; see \cite[Prop.~2.4]{GrLe}.

The modules $\{W^\lambda\mid \lambda \in \Lambda\}$ are called the
{\em cell representations}, or {\em cell modules}, of $\cH$. Extending 
scalars from $A$ to $K$, we obtain modules $W_K^\lambda=K\otimes_A 
W^\lambda$ for $\cH_K$. By the discussion in \cite[Exp.~4.4]{mycell}, we have 
\[ \Irr(\cH_K)=\{W_K^\lambda \mid \lambda \in \Lambda\}
\quad \mbox{and} \quad W_K^\lambda\cong E^\lambda_v \quad 
\mbox{for all $\lambda \in \Lambda$}.\]
Now let $\theta \colon A \rightarrow k$ be a ring homomorphism into a 
field $k$; note that the characteristic of $k$ will be either $0$ or a 
prime $p$ which is not bad for $W$. By extension of scalars, we obtain a 
$k$-algebra $\cH_k(W,\xi) =k \otimes_A \cH$ where $\xi:=\theta(u)\in k$. 
Explicitly, $\cH_k(W,\xi)$ has a basis $\{T_w \mid w\in W\}$ and the
multiplication is given by 
\[ T_sT_w=\left\{\begin{array}{cl} T_{sw} &\quad\mbox{if $l(sw)=l(w)+1$},\\
\xi T_{sw}+(\xi-1)T_w & \quad \mbox{if $l(sw)=l(w)-1$},
\end{array} \right.\]
where $s\in S$ and $w\in W$. The algebra $\cH_{k}(W,\xi)$ is called a 
{\em specialisation} of $\cH$. Let $\Irr(\cH_{k}(W,\xi))$ be the set of 
irreducible representations of $H_{k}(W,\xi)$, up to isomorphism. 

Now, we also obtain cell modules $W_\xi^\lambda=k \otimes_A 
W^\lambda$ ($\lambda \in \Lambda$) for $\cH_k(W,\xi)$, which may no longer 
be irreducible. Denoting by $\phi_\xi^\lambda$ the induced bilinear form on 
$W_\xi^\lambda$, we set 
\[ L_\xi^\lambda=W_\xi^\lambda/\mbox{rad}(\phi_\xi^\lambda).\] 
Then, by the general theory of cellular algebras in \cite[\S 3]{GrLe}, each 
$L_\xi^\lambda$ is either $\{0\}$ or an absolutely simple 
$\cH_k(W,\xi)$-module, and we have
\[ \Irr(\cH_k(W,\xi))=\{L_\xi^\mu \mid \mu \in \Lambda_\xi^\circ\} \qquad 
\mbox{where} \qquad \Lambda_\xi^\circ:=\{\lambda \in \Lambda \mid 
L_\xi^\lambda \neq 0\}.\]
In particular, this shows that the algebra $\cH_k(W,\xi)$ is split. 
Furthermore, denoting by $(W_\xi^\lambda:L_\xi^\mu)$ the multiplicity of 
$L_\xi^\mu$ as a composition factor of $W_\xi^\lambda$, we have
\begin{equation*}
\left\{\begin{array}{l} \quad (W_\xi^\mu:L_\xi^\mu)=1 \quad \mbox{for any
$\mu\in \Lambda_\xi^\circ$},\\[1mm] \quad (W_\xi^\lambda:L_\xi^\mu)=0\quad
\mbox{unless $\lambda=\mu$ or $\ba_\mu< \ba_\lambda$}. \end{array}\right. 
\tag{$\Delta$}
\end{equation*}
Thus, the theory of cellular algebras provides a general method for
constructing the irreducible representations of the specialized algebra 
$\cH_k(W,\xi)$.

\medskip
\begin{prop} \label{prop11} Assume that $P_W(\xi)\neq 0$. Then 
$\cH_k(W,\xi)$ is semisimple, $\Lambda=\Lambda_\xi^\circ$ and 
$W_\xi^\lambda=L_\xi^\lambda$ for all $\lambda \in \Lambda$.
\end{prop}

\begin{proof} Recall from Remark~\ref{equalp} that, for each $\lambda
\in \Lambda$, we have $\bc_\lambda=f_\lambda u^{-\ba_\lambda}\, 
\tilde{\bc}_\lambda$ where $\tilde{\bc}_\lambda \in {\Z}[u]$ is monic 
and divides $P_W$. Hence, since the characteristic of $k$ is either $0$ 
or a good prime for $W$, our assumption $P_W(\xi)\neq 0$ implies that we 
also have $\theta(\bc_\lambda) \neq 0$ for all $\lambda \in \Lambda$. A 
general semisimplicity criterion for symmetric algebras (see 
\cite[7.4.7]{ourbuch}) then shows that $\cH_k(W,\xi)$ is semisimple, a 
result first proved by Gyoja--Uno \cite{GyUn}. The remaining statements
concerning the cell representations are contained in \cite[3.8]{GrLe}.  
\end{proof}

\medskip
\begin{cor} \label{detg} Let $\lambda \in \Lambda$ and $G^\lambda$ be the
Gram matrix of the invariant bilinear form $\phi^\lambda$ with respect to
the standard basis of $W^\lambda$. Then $0\neq \det(G^\lambda) 
\in {\Z}[v,v^{-1}]$. Furthermore, let $0 \neq q \in {\Z}[v,v^{-1}]$ be 
irreducible such that $q$ divides $\det(G^\lambda)$. Then either $\pm q$ 
is a bad prime number or $q$ divides $P_W$.
\end{cor}

\begin{proof} First note that, by Remark~\ref{integral}, all entries of 
$G^\lambda$ lie in ${\Z}[v,v^{-1}]$. Furthermore, by Proposition~\ref{prop11}, we have $\det(G^\lambda)\neq 0$. Now consider the prime ideal $(q)$ 
and let $F$ be the field of fractions of $A/(q)$. Then we have a 
specialisation $\alpha \colon A \rightarrow F$. Let $\cH_F(W,\alpha(u))$ 
be the specialised algebra. Let $G^\lambda_F$ be the matrix obtained by 
applying $\alpha$ to all coefficients of $G^\lambda$. Then $G^\lambda_F$ 
is the Gram matrix of the induced bilinear form $\phi^\lambda_F$ on the 
specialised cell module $W^\lambda_F$. If $q$ divides $\det(G^\lambda)$, 
then $\det(G_F^\lambda)=0$ and so $\cH_F(W,\alpha(u))$ will not be 
semisimple; see \cite[3.8]{GrLe}. By the general semisimplicity criterion 
in \cite[7.4.7]{ourbuch}, we deduce that $\alpha(\bc_\mu)=0$ for some 
$\mu \in \Lambda$. Now there are two cases.

If $q \in \Z$, then this implies that $q$ must divide $f_\mu$ and
so $\pm q$ is a bad prime. 

If $q$ is an irreducible non-constant polynomial, then $q$ must divide
$\bc_\mu$. By Remark~\ref{equalp}, $\bc_\mu$ divides $P_W$. Hence, we 
deduce that $q$ divides $P_W$. 
\end{proof}

\medskip
\begin{exmp} \label{typeA} Let $W$ be of type $A_{n-1}$. Then $W$ can  be
identified with the symmetric group $\fS_n$ and $\Lambda$ consists of all 
partitions $\lambda \vdash n$. A special feature of this case is that 
$f_\lambda=1$ for all $\lambda \in \Lambda$. By \cite[Exp.~4.2]{mycell}, 
the linear combinations in Theorem~\ref{Hcellbase} will only have one
non-zero term, with coefficient $1$, i.e., the Kazhdan--Lusztig basis 
itself is a cellular basis. More precisely, for $\lambda \in \Lambda$, 
let $w_\lambda$ be the longest element in the corresponding Young 
subgroup $\fS_\lambda$ of $W=\fS_n$. Now, by \cite[\S 5]{KaLu}, the 
Kazhdan--Lusztig left and right cells of $W$ are given by the 
Robinson--Schensted correspondence. This explicit description shows that, 
if $\fC_\lambda$ 
denotes the left cell containing $w_\lambda$, we have  
\[ \fC_\lambda=\{d(\fs)w_\lambda \mid \fs \in M(\lambda)\}\]
where the elements $d(\fs)$ ($\fs \in M(\lambda)$) are certain distinguished 
left coset representatives of $\fS_\lambda$ in $W=\fS_n$. Furthermore, given
$\fs,\ft \in M(\lambda)$, there is a unique $w_\lambda(\fs,\ft) \in W$ such 
that $w_\lambda(\fs,\ft)$ lies in the same right cell as $d(\fs)w_\lambda$ 
and in the same left cell as $w_\lambda d(\ft)^{-1}$. (See also 
\cite[Rem.~3.9, Cor.~5.6]{mymurph} for further details.) With this 
notation, \cite[Exp.~4.2]{mycell} shows that 
\[ C_{\fs,\ft}^\lambda=c_{w_\lambda(\fs,\ft)}^\dagger \quad \mbox{for all
$\lambda \vdash n$ and $\fs,\ft \in M(\lambda)$}.\]
McDonough and Pallikaros \cite{McPa} showed that the cell modules 
$W^\lambda$ are naturally isomorphic to the Dipper--James Specht modules.
The invariant bilinear form on $W^\lambda$ is given by 
\[ \phi^\lambda(C_{\fs},C_{\ft})=h_{w_\lambda d(\fs)^{-1}, d(\ft)w_\lambda,
w_\lambda} \qquad \mbox{for all $\fs,\ft \in M(\lambda)$}.\]
For connections of these bilinear forms with the topology of Springer fibres,
see Fung \cite{Fung}. 

Thus, for general $\cH$, the cell modules $W^\lambda$ arising from 
Theorem~\ref{Hcellbase} can indeed be regarded as analogues of the 
Dipper--James Specht modules in type $A_{n-1}$.
\end{exmp}

\medskip
\begin{exmp} \label{expG2} Let $W$ be the Weyl group of type $G_2$
where $S=\{s_1,s_2\}$ and $(s_1s_2)^6=1$. We have 
$\Irr(W)=\{{\bf 1},\varepsilon_1,\varepsilon_2, \varepsilon,r,r'\}$
where ${\bf 1}$ is the unit representation, $\varepsilon$ is the sign 
representation,  $\varepsilon_1$, $\varepsilon_2$ have dimension one, $r$
is the reflection representation and $r'$ is another representation of 
dimension two. The invariants $\ba_\lambda$ and $f_\lambda$ are given by
\begin{gather*}
\ba_{\bf 1}=0, \qquad \ba_{\varepsilon_1}=\ba_{\varepsilon_2}=\ba_r=
\ba_{r'}=1, \qquad \ba_{\varepsilon}=6;\\
f_{\bf 1}=f_{\varepsilon}=1,\qquad f_{\varepsilon_1}=f_{\varepsilon_2}=3,
\quad f_r=6, \qquad f_{r'}=2.
\end{gather*}
Hence, the bad primes are $2$ and $3$. A cellular basis as in 
Theorem~\ref{Hcellbase} is given as follows:
\begin{alignat*}{2}
C_{1,1}^{\bf 1}&=c_1^\dagger, \qquad & C_{1,1}^\varepsilon &= 
c_{w_0}^\dagger, \\
C_{1,1}^{\varepsilon_1}&= c_{s_2}^\dagger- c_{s_2s_1s_2}^\dagger+
c_{s_2s_1s_2s_1s_2}^\dagger,\qquad &
C_{1,1}^{\varepsilon_2}&=c_{s_1}^\dagger- c_{s_1s_2s_1}^\dagger+
c_{s_1s_2s_1s_2s_1}^\dagger,\\
C_{1,1}^r &= 3c_{s_1}^\dagger+6c_{s_1s_2s_1}^\dagger+
3c_{s_1s_2s_1s_2s_1}^\dagger,
\qquad & C_{1,1}^{r'}&=c_{s_1}^\dagger-c_{s_1s_2s_1s_2s_1}^\dagger,\\
C_{1,2}^r &= -3c_{s_1s_2}^\dagger-3c_{s_1s_2s_1s_2}^\dagger, \qquad &
C_{1,2}^{r'} &= -c_{s_1s_2}^\dagger+c_{s_1s_2s_1s_2}^\dagger,
\\ C_{2,1}^r &=-3c_{s_2s_1}^\dagger-3c_{s_2s_1s_2s_1}^\dagger, 
\qquad & C_{2,1}^{r'}&=-c_{s_2s_1}^\dagger+c_{s_2s_1s_2s_1}^\dagger,\\ 
C_{2,2}^r &= c_{s_2}^\dagger+2c_{s_2s_1s_2}^\dagger+
c_{s_2s_1s_2s_1s_2}^\dagger, \qquad
& C_{2,2}^{r'}&=c_{s_2}^\dagger-c_{s_2s_1s_2s_1s_2}^\dagger.
\end{alignat*}
To find these expressions, we perform computations similar to those in 
\cite[Exp.~4.3]{mycell} (where type $B_2$ was considered). Once this is 
done, one can then also check directly that the above elements form a 
cellular basis. The Gram matrices of the invariant bilinear forms on the
cell representations $W^\lambda$ are given by 
\begin{gather*}
G^{\bf 1}=\begin{bmatrix} \,1\,\end{bmatrix},\quad 
G^{\varepsilon}=\begin{bmatrix} v^{-6}P_W \end{bmatrix},
\quad G^{\varepsilon_1}=G^{\varepsilon_2}=
\begin{bmatrix} 3(v+v^{-1})\end{bmatrix}, \\
G^{r}=\left[\begin{array}{cc} 18(v+v^{-1}) & -18 \\ -18& 6(v+v^{-1})
\end{array}\right],\qquad
G^{r'}=\left[\begin{array}{cc} 2(v+v^{-1}) & -2 \\ -2 & 2(v+v^{-1})
\end{array}\right],
\end{gather*}
where $P_W=(v^{12}-1)(v^4-1)/(v^2-1)^2$ is the Poincar\'e polynomial of $W$.

Now let $\theta \colon A \rightarrow k$ be a specialisation; note that the 
characteristic of $k$ will be either $0$ or a prime $\neq 2,3$. Let $e\geq 
2$ be minimal such that $1+\xi+\xi^2+\cdots+\xi^{e-1}=0$. Thus, either 
$\xi=1$ and $e$ is the characteristic of $k$, or $e$ is the multiplicative 
order of $\xi$ in $k^\times$. We see that the above Gram matrices remain 
non-singular after specialisation unless $\xi \neq 1$ and $e\in \{2,3,6\}$. 
Thus, we obtain non-trivial decomposition numbers only for $e\in\{2,3,6\}$. 
In these cases, the sets $\Lambda^\circ_\xi$ and the dimensions of 
$L^\mu_\xi$ for $\mu \in \Lambda^\circ_\xi$ are given as follows.
\[ \renewcommand{\arraystretch}{1.2}
\begin{array}{ccc} \hline \multicolumn{3}{c}{e=2} \\ \hline 
\Lambda^\circ_\xi & \quad \ba_\mu \quad & \dim L^\mu_\xi \\ \hline
{\bf 1} & 0 & 1 \\ 
r       & 1 & 2 \\
r'      & 1 & 2 \\ \hline &&\\ && \end{array}\qquad\quad
 \begin{array}{ccc} \hline \multicolumn{3}{c}{e=3} \\ \hline 
\Lambda^\circ_\xi & \quad \ba_\mu \quad & \dim L^\mu_\xi \\ \hline
{\bf 1}       & 0 & 1 \\ 
\varepsilon_1 & 1 & 1 \\
\varepsilon_2 & 1 & 1 \\
r             & 1 & 2 \\ 
r'            & 1 & 1 \\ \hline\end{array}\qquad\quad
\begin{array}{ccc} \hline \multicolumn{3}{c}{e=6} \\ \hline 
\Lambda^\circ_\xi & \quad \ba_\mu \quad & \dim L^\mu_\xi \\ \hline
{\bf 1}       & 0 & 1 \\ 
\varepsilon_1 & 1 & 1 \\
\varepsilon_2 & 1 & 1 \\
r             & 1 & 1 \\ 
r'            & 1 & 2 \\ \hline\end{array}\]
In particular, we notice that the classification of the irreducible 
representations and their dimensions only depend on~$e$, but not on the 
particular value of $\xi$ or the characteristic of~$k$.  Thus, we have 
verified in a particular example the general phenomenon which is expressed 
in James' conjecture.
\end{exmp}

\medskip
\begin{rem} \label{brauer} The decomposition matrix $D_\xi$ can also
be interpreted in the framework of Brauer's modular representation theory
of associative algebras; see \cite[\S I.1.17]{Feit}. Indeed, let us assume
that $k$ is the field of fractions of the image of $\theta$. By 
\cite[Exc.~7.8]{ourbuch}, there exists a discrete valuation ring $\cO 
\subseteq K$ with maximal ideal $\fp$ such that $A \subseteq \cO$ and
$\fp \cap A=\ker(\theta)$. Let $k_{\fp} \supseteq k$ be the residue field 
of $\cO$. Since $\cH_k(W,\xi)$ is split, the scalar extension from $k$ to
$k_\fp$ induces a bijection $\Irr(\cH_k(W,\xi))\stackrel{\sim}{\rightarrow}
\Irr(\cH_{k_{\fp}}(W,\xi))$. Identifying $\Irr(\cH_k(W,\xi))$ and 
$\Irr(\cH_{k_\fp}(W,\xi))$ via this isomorphism, we obtain a well-defined
decomposition map 
\[ d_\xi \colon R_0(\cH_K)\rightarrow R_0(\cH_k(W,\xi))\]
where $R_0(\cH_K)$ and $R_0(\cH_k(W,\xi))$ denote the Grothendieck groups 
of finite-dimensional representations of $\cH_K$ and $\cH_k(W,\xi)$, 
respectively. Since each cell representation $W^\lambda$ is defined over 
$A$ and $W^\lambda_K \cong E^\lambda_v$, we conclude that 
\[ d_\xi([E^\lambda_v])=\sum_{\mu \in \Lambda_\xi^\circ} (W_\xi^\lambda: 
L^\mu_{\xi}) \, [L^\mu_\xi] \qquad \mbox{for all $\lambda\in \Lambda$},\]
where $[E^\lambda_v]$, $[L^\mu_\xi]$ denote the classes of $E^\lambda_v$, 
$L^\mu_\xi$ in the respective Grothendieck groups.
(Note that, by \cite[Ex. 6.16]{CR1}, we do not need to pass to the completion
of $\cO$, as is usually done in Brauer's modular representation theory.)
\end{rem}

\medskip
\begin{defn} \label{Bgraph} The {\em Brauer graph} of $\cH$ with respect 
to $\theta \colon A \rightarrow k$ is the graph with vertices labelled by 
the elements of $\Lambda$ and edges given as follows. Let $\lambda \neq
\lambda'$ in $\Lambda$. Then the vertices labelled by $\lambda$ and 
$\lambda'$ are joined by an edge if there exists some $\mu \in 
\Lambda^\circ_\xi$ such that $(W_\xi^\lambda:L^\mu) \neq 0$ and 
$(W_\xi^{\lambda'}:L^\mu)\neq 0$. The connected components of this graph 
define a partition of $\Lambda$ which are called the {\em $\xi$-blocks} 
of $\Lambda$ (or of $\Irr(\cH_K)$ or of $\Irr(W)$). 
\end{defn}

\medskip
Let $\Lambda=\Lambda_1 \amalg \Lambda_2 \amalg \cdots \amalg \Lambda_r$
be the partition of $\Lambda$ into $\xi$-blocks. Then we also have
\[ \Lambda_\xi^\circ=\Lambda_{\xi,1}^\circ \amalg \Lambda_{2,\xi}^\circ 
\amalg \cdots \amalg \Lambda_{\xi,r}^\circ\qquad \mbox{where}\qquad 
\Lambda_{\xi,i}^\circ:=\Lambda_i \cap \Lambda_\xi^\circ.\]
If we order the elements of $\Lambda$ and of $\Lambda_\xi^\circ$ accordingly, 
we obtain a block diagonal shape for $D_\xi$:
\[ D_\xi=\left(\begin{array}{cccc} D_{\xi,1} & 0 & \ldots & 0 \\
0 & D_{\xi,2} & \ddots & \vdots \\ \vdots & \ddots & \ddots & 0 \\
0 & \ldots & 0 & D_{\xi,r} \end{array}\right),\]
where $D_{\xi,i}$ has rows and columns labelled by the elements of 
$\Lambda_i$ and $\Lambda_{\xi,i}^\circ$, respectively. Thus, in order to 
describe the set $\Lambda_{\xi}^\circ$ and the matrix $D_\xi$, we can 
proceed block by block. Note that, by Remark~\ref{brauer}, the blocks
of $\cH$ as defined above really correspond to blocks in the sense of
Brauer's modular representation theory.

\section{The general version of James' conjecture} \label{sec3}
We keep the general setting of the previous section. Let $\cH$ be an 
Iwahori--Hecke algebra associated with a finite Weyl group $W$, defined 
over the ring $A=R[v,v^{-1}]$ where $R \subseteq \Q$ is fixed as in the 
remarks just after Theorem~\ref{Hcellbase}. Then we have a cellular basis 
$\{C_{\fs,\ft}^\lambda\}$ and cell representations $\{W^\lambda \mid 
\lambda \in \Lambda\}$ for $\cH$. 

Now let $\theta \colon A \rightarrow k$ be a ring homomorphism into a 
field $k$. Note that the characteristic of $k$ will be either $0$ or a 
prime $p$ which is not bad for $W$. We obtain a corresponding specialised 
algebra $\cH_k(W,\xi)$ where $\xi=\theta(u) \in k^\times$. Recall that 
\[ \Irr(\cH_{k}(W,\xi))=\{L_{\xi}^\mu \mid \mu \in
\Lambda_{\xi}^\circ\}.\]
As in Remark~\ref{brauer}, we have a decomposition map
$d_{\xi} \colon R_0(\cH_K)\rightarrow R_0(\cH_k(W,\xi))$
such that
\[ d_{\xi}([E_v^\lambda])=\sum_{\mu \in \Lambda_{\xi}^\circ}
(W_{\xi}^\lambda: L^\mu_{\xi}) \, [L^\mu_{\xi}] \qquad
\mbox{for all $\lambda\in \Lambda$}.\]
Following Dipper--James \cite{DJ}, we set 
\[ e=\min\{i \geq 2\mid 1+\xi+\xi^2+\cdots +\xi^{i-1}=0\}.\]
(We set $e=\infty$ if no such $i$ exists.) 
We assume from now on that $\mbox{char}(k)=\ell>0$ and $e<\infty$. Let 
$\zeta_e=\sqrt[e]{1} \in \C$ and consider the Iwahori--Hecke algebra 
$\cH_{\C}(W,\zeta_e)$ arising from the specialisation 
\[\theta_e\colon A\rightarrow {\C},\quad v\mapsto\zeta_{2e}=\sqrt[2e]{1}.\]
We can apply the previous discussion to the algebra $\cH_{\C}(W,
\zeta_e)$ as well. Thus, we have 
\[ \Irr(\cH_{\C}(W,\zeta_e))=\{L_{\zeta_e}^\mu \mid \mu \in 
\Lambda_{\zeta_e}^\circ\}.\]
Furthermore, there is a decomposition map $d_{\zeta_e} \colon 
R_0(\cH_K)\rightarrow R_0(\cH_\C(W,\zeta_e))$ such that 
\[ d_{\zeta_e}([E_v^\lambda])=\sum_{\mu \in \Lambda_{\zeta_e}^\circ}
(W_{\zeta_e}^\lambda: L^\mu_{\zeta_e}) \, [L^\mu_{\zeta_e}] \qquad 
\mbox{for all $\lambda\in \Lambda$}.\]
We will want to compare the representations of $\cH_k(W,\xi)$ and 
$\cH_{\C}(W,\zeta_e)$. For this purpose, the following remark 
will be relevant.

\medskip
\begin{rem} \label{compare} For any $d\geq 1$, we denote by
$\Phi_d\in {\Z}[u]$ the $d$th cyclotomic polynomial. Note that we have
\[ \Phi_{d}(v^2)=\left\{\begin{array}{cl} \Phi_{2d}(v) & \qquad \mbox{if
$d$ is even},\\\Phi_d(v)\Phi_d(-v) &\qquad \mbox{if $d$ is odd}.
\end{array}\right.\]
Now, in view of the definition of $e$, it is clear that $\Phi_e(\xi)=0$.
Furthermore, note that $\theta(v)^2=\xi$. Hence, choosing a square root of 
$\xi$ in $k^\times$ appropriately, we can assume that $\Phi_{2e}(\theta(v))
=0$. (If $\mbox{char}(k)\neq 2$, we also have $\Phi_e(\theta(v))\neq 0$.)
Consequently, there exists a ring homomorphism $R[\zeta_{2e}]\rightarrow k$,
$r\mapsto \bar{r}$, such that $\theta(a)=\overline{\theta_e(a)}$ 
for all $a\in A$. Let $\cO\subseteq {\Q}(\zeta_{2e})$ be the localisation 
of $R[\zeta_{2e}]$ in the prime ideal $\fq=\{r \in R[\zeta_{2e}] \mid 
\bar{r}=0\}$. Then $\cO$ is a discrete valuation ring whose residue field 
can be identified with a subfield of $k$. By ``$\fq$-modular reduction'' 
(see \cite[\S I.1.17]{Feit}), we obtain a well-defined decomposition map 
\[ d_{\xi}^e \colon R_0(\cH_{\Q(\zeta_{2e})}(W,\zeta_e))\rightarrow 
R_0(\cH_k(W,\xi)).\]
Note that the scalar extension from ${\Q}(\zeta_{2e})$ to $\C$ defines a 
bijection 
\[ \Irr(\cH_{\Q(\zeta_{2e})}(W,\zeta_e)) \stackrel{\sim}{\rightarrow}
\Irr(\cH_{\C}(W,\zeta_e)).\]
Via this bijection, we can identify $R_0(\cH_{\Q(\zeta_{2e})}(W,\zeta_e))$ 
and $R_0(\cH_{\C}(W,\zeta_e))$, and regard $d_\xi^e$ as a map from
$R_0(\cH_{\C}(W,\zeta_e))$ to $R_0(\cH_k(W,\xi))$. Let us write
\[ d_{\xi}^e([L^\nu_{\zeta_e}])=\sum_{\mu \in \Lambda_{\xi}^\circ} 
a_{\nu\mu} \, [L_\xi^\mu] \qquad \mbox{for any $\nu \in 
\Lambda_{\zeta_e}^\circ$},\]
where $a_{\nu\mu}\in {\Z}_{\geq 0}$.
Following James \cite{Ja}, the matrix $A_\xi^e:=(a_{\nu\mu})$ is called the 
{\em adjustment matrix} associated to the specialisation $\theta$. By
a general factorisation result for decomposition maps, we have $d_\xi=
d_{\xi}^e \circ d_{\zeta_e}$ or, in other words,

\smallskip
\begin{center}
\fbox{$\displaystyle (W_\xi^\lambda:L_\xi^\mu)=\sum_{\nu \in 
\Lambda_{\zeta_e^\circ}} a_{\nu\mu}\, (W^\lambda_{\zeta_e}:L_{\zeta_e}^\nu)
\quad\mbox{for all $\lambda\in\Lambda$ and $\mu\in\Lambda_\xi^\circ$}$.}
\end{center}
\smallskip
This result first appeared in \cite[Theorem~5.3]{mybaum}; see also
\cite[Prop.~2.5]{GeRo1}, \cite[Prop.~2.6]{mybour} for analogous statements
in more general situations.
\end{rem}

\medskip
\begin{lem} \label{factor} In the above setting, the following hold.
\begin{description}
\item[(a)] Given $\mu \in \Lambda_\xi^\circ$ and $\nu \in 
\Lambda_{\zeta_e}^\circ$, we have $a_{\nu\mu}=0$ unless $\nu=\mu$ or 
$\ba_\mu<\ba_\nu$.
\item[(b)] We have $\Lambda_\xi^\circ \subseteq \Lambda_{\zeta_e}^\circ$ and
$a_{\mu\mu}=1$ for all $\mu \in \Lambda_\xi^\circ$.
In particular, we have $\Lambda_\xi^\circ=\Lambda_{\zeta_e}^\circ$ if 
these two sets have the same cardinality.
\item[(c)] We have  $\dim L_\xi^\mu\leq\dim L_{\zeta_e}^\mu$ 
for all $\mu\in \Lambda_\xi^\circ$. 
\end{description}
\end{lem}

\begin{proof} Let $\lambda \in \Lambda$, $\mu \in \Lambda_{\xi}^\circ$
and $\nu \in \Lambda_{\zeta_e}^\circ$. Recall the relations ($\Delta$) 
from Section~2: if $(W_\xi^\lambda:L_\xi^\mu)\neq 0$, then $\ba_\mu\leq 
\ba_\lambda$ with equality only for $\lambda=\mu$; furthermore,
$(W_\xi^\mu:L_\xi^\mu)=1$. A similar statement holds for the decomposition
numbers $(W_{\zeta_e}^\lambda:L_{\zeta_e}^\nu)$.

(a) Assume that $a_{\nu\mu}\neq 0$. Then, since 
$(W^\nu_{\zeta_e}:L^\nu_{\zeta_e})=1$, we have 
\[ (W_\xi^\nu: L_\xi^\mu)=\sum_{\nu' \in \Lambda_{\zeta_e}^\circ} 
a_{\nu'\mu}\, (W^\nu_{\zeta_e}:L_{\zeta_e}^{\nu'})>0\]
and so the relations ($\Delta$) imply that $\nu=\mu$ or $\ba_\mu<\ba_\nu$. 

(b) We have $1=(W^\mu_\xi:L^\mu_\xi)=\sum_{\nu' \in \Lambda_{\zeta_e}^\circ}
a_{\nu'\mu}\, (W^\mu_{\zeta_e}:L_{\zeta_e}^{\nu'})$.
So there exists some $\nu'\in \Lambda_{\zeta_e}^\circ$ such that 
$a_{\nu'\mu}\neq 0$ and $(W^\mu_{\zeta_e}:L_{\zeta_e}^{\nu'})\neq 0$.
Consequently, using (a) and the relations ($\Delta$), we have $\ba_\mu
\leq \ba_{\nu'} \leq \ba_{\mu}$ and so $\ba_{\mu}=\ba_{\nu'}$. Thus, we 
must have $\mu=\nu' \in \Lambda_{\zeta_e}^\circ$ and $a_{\mu\mu} \neq 0$. 
Since $(W^\mu_\xi:L^\mu_\xi)=1$, we then also conclude that $a_{\mu\mu}=1$.

(c) Since $\dim L_{\zeta_e}^\mu=\sum_{\nu\in\Lambda_\xi^\circ} a_{\mu\nu}
\dim L_\xi^\nu \geq a_{\mu\mu}\dim L_\xi^\mu$, this follows from (b).
\end{proof}

The observation that $\Lambda_\xi^\circ$ equals $\Lambda_{\zeta_e}^\circ$ 
once we know that these two sets have the same cardinality was first made 
by Jacon \cite[Theorem~3.3]{Jac} (in a slightly different context). 

\medskip
\begin{thm}[Geck--Rouquier \protect{\cite[5.4]{GeRo1}, 
\cite[3.2]{my00b}}] \label{thm13} Assume that $e\ell$ does not divide any 
degree of $W$. Then $|\Irr(\cH_k(W,\xi))|=|\Irr(\cH_{\C}(W,\zeta_e))|$.
\end{thm}

\medskip
Actually, using some explicit computations for $W$ of exceptional type and 
the results of Ariki--Mathas \cite{AM} for $W$ of classical type, one can 
show that the above conclusion holds under the single assumption that $\ell$ 
is a good prime; see \cite{my00b}. However, we do not need this stronger
result here.

\medskip
\begin{rem} \label{signi} The significance of the assumption on $\ell$
in Theorem~\ref{thm13} is as follows. One easily checks that if $f \geq 2$
is such that $\Phi_f(\xi)=0$ then $f=e\ell^i$ for some $i \geq 0$ (see, for 
example, \cite[3.1]{my00b}). Hence, assuming that $e\ell$ does not divide 
any degree of $W$, we have the following implication for any $f \geq 2$:
\[ \Phi_f(\xi)=0 \mbox{ and } \Phi_f \mbox{ divides } P_W \quad 
\Rightarrow \quad f=e.\]
\end{rem}

\medskip
\begin{conj}[General version of James' conjecture] \label{Jconj} 
Recall our standing assumption that $e<\infty$ and $\operatorname{char}
(k)=\ell>0$ where $\ell$ is a good prime for $W$. Assume also that $e\ell$ 
does not divide any degree of $W$. Then the decomposition matrix $D_\xi$ only 
depends on $e$. More precisely, the adjustment matrix $A_\xi^e$ is the
identity matrix or, in other words:
\begin{equation*}
(W_\xi^\lambda:L^\mu_\xi)=(W^\lambda_{\zeta_e}:L^\mu_{\zeta_e})
\qquad \mbox{for all $\;\lambda \in \Lambda \;$ and $\; \mu\in 
\Lambda_\xi^\circ= \Lambda_{\zeta_e}^\circ$}.\tag{J}
\end{equation*}
(Note that we do know that $\Lambda_\xi^\circ=\Lambda_{\zeta_e}^\circ$ by 
Theorem~\ref{thm13} and Lemma~\ref{factor}.)
\end{conj}

\medskip
Using the factorisation in Remark~\ref{compare} and Lemma~\ref{factor}, 
the above conjecture can be reformulated as follows. 

\medskip
\begin{cor}[Alternative version of James' Conjecture] \label{Jconj1} 
Condition {\rm (J)} in Conjecture~\ref{Jconj} holds if and only if $\dim 
\operatorname{rad}(\phi_\xi^\lambda)= \dim \operatorname{rad}
(\phi_{\zeta_e}^\lambda)$ for all $\lambda \in \Lambda$.
\end{cor}

\medskip
Thus, in order to verify James' conjecture, it is sufficient to determine 
the ranks of the Gram matrices of the bilinear forms $\phi^\lambda$ for 
various specialisations. Recall from Section~2 that the entries of these
Gram matrices are certain structure constants of $\cH$ with respect to its
cellular basis, and these can be expressed in terms of the structure
constants of the Kazhdan--Lusztig basis of $\cH$. These in turn can be 
computed in principle (using recursive formulae), but note that this is only
feasible for algebras of small rank. In Section~\ref{sec4} and \cite{gemu2}, 
we will see how this problem can be solved effectively.

\medskip
\begin{prop}[See also \protect{\cite[Prop.~5.5]{mybaum}} and
\protect{\cite[2.7]{mybour}}] \label{largeell} There exists a bound $N$, 
depending only on $W$, such that condition {\rm (J)} in 
Conjecture~\ref{Jconj} holds for all $\ell>N$.
\end{prop}

\begin{proof} We introduce the following notation. Given any matrix $M$ with 
entries in $A$, we denote by $M_\xi$ the matrix obtained by applying 
$\theta$ to all entries of $M$. Similarly, we define $M_{\zeta_{e}}$ via 
the map $\theta_e$; the entries of $M_{\zeta_e}$ will lie in $R[\zeta_{2e}]$.
Finally, if $N$ is a matrix with entries in $R[\zeta_{2e}]$, we denote by 
$\bar{N}$ the matrix obtained by applying the map $\alpha \mapsto 
\bar{\alpha}$ to all entries of $N$ (see Remark~\ref{compare}). With this 
notation, we have $M_\xi= \overline{M}_{\zeta_{e}}$ for any matrix $M$ 
with entries in $A$.

Now fix $e\geq 2$ and $\lambda \in \Lambda$. Let $G^\lambda$ be the Gram 
matrix of $\phi^\lambda$; this is a matrix with entries in ${\Z}[v,v^{-1}]$.
With the above notation, we have $G^\lambda_\xi=
\overline{G}_{\zeta_e}^\lambda$. This already implies that $\mbox{rank} 
(G_\xi^\lambda) \leq r:=\mbox{rank}(G_{\zeta_{e}}^\lambda)$. We can find 
an $r\times r$-submatrix $G$ of $G^\lambda$ such that $\det(G_{\zeta_e})
\neq 0$. Now $\det(G_{\zeta_e})$ is an algebraic integer in the ring 
${\Z}[\zeta_{2e}]$; its norm will be a non-zero rational integer. If 
$\ell$ does not divide that integer, we have 
\[ \det(G_\xi)=\det(\overline{G}_{\zeta_e})=\overline{\det(G_{\zeta_e})} 
\neq 0.\]
So $r=\mbox{rank}(G_\xi^\lambda)=\mbox{rank}(G_{\zeta_e}^\lambda)$ for 
$\ell$ ``large enough''. Hence, since $\Lambda$ is a finite set, there is 
global bound $N$ such that $\mbox{rank}(G_\xi^\lambda)=\mbox{rank}
(G_{\zeta_e}^\lambda)$ for all $\lambda \in \Lambda$ and all $\ell>N$. 
Hence, by Corollary~\ref{Jconj1}, the conclusion of James' conjecture 
holds for all $\ell>N$, 
\end{proof}

Note that the above proof actually provides a method for finding $N$,
assuming that the Gram matrices $G^\lambda$ are explicitly known.

Recall from Section~2 the definition of the Brauer graph of $\cH$ with 
respect to $\theta \colon A \rightarrow k$; its connected components are 
called {\em $\xi$-blocks}. Similarly, we define the Brauer graph of $\cH$ 
with respect to $\theta_e \colon A \rightarrow \C$. Its connected 
components are called {\em $\zeta_e$-blocks}. 

\medskip
\begin{defn} \label{defect} Given $\lambda \in \Lambda$, we set
\[ \delta_\lambda:=\max\{i \geq 0 \mid \Phi_e^i \mbox{ divides } 
\bc_\lambda \mbox{ in ${\Q}[u]$}\}.\]
This number is called the {\em $\Phi_e$-defect} of $\lambda$ (or of 
$E^\lambda$). 
\end{defn}

\medskip
\begin{prop}[Geck \protect{\cite[7.4 and 7.6]{mybaum}}] \label{block}
Assume that $e\ell$ does not divide any degree of $W$. Then the following 
hold.
\begin{description}
\item[(a)] The $\xi$-blocks of $\cH$ coincide with the $\zeta_e$-blocks 
of $\cH$. 
\item[(b)] If $E^\lambda$ and $E^\mu$ belong to the same $\xi$-block, then 
$\delta_\lambda=\delta_\mu$.
\end{description}
\end{prop}

\medskip
The above result shows that all irreducible representations in a given 
$\xi$-block of $\cH$ have the same $\Phi_e$-defect, which will be called the 
$\Phi_e$-defect of the block. Note that the only known proof of 
Proposition~\ref{block}(b) relies on an interpretation of $D_{\zeta_e}$ in 
the modular representation theory of a finite group of Lie type with Weyl 
group $W$, and on known results on heights of characters in blocks of finite 
groups with abelian defect groups.

\begin{table}[htbp] \caption{The sets $\Lambda^\circ_{\zeta_e}$ for type 
$F_4$, $E_6$, $E_7$} \label{canf4}
\[ \renewcommand{\arraystretch}{1.1}
\begin{array}{l} 
\begin{array}{ccc}\hline \multicolumn{3}{c}{F_4,e=2}\\\hline 
1_1 & 0 & 1 \\ 2_1 & 1 & 2 \\ 2_3 & 1 & 2 \\ 9_1 & 2 & 5 \\ \hline &&
\end{array}\qquad
\begin{array}{ccc}\hline \multicolumn{3}{c}{F_4,e=2}\\\hline
4_2 & 1 & 4 \\ \hline && \\&&\\&&\\&&\end{array} \qquad
\begin{array}{ccc} \hline \multicolumn{3}{c}{F_4,e=3}\\\hline 
1_1 & 0 & 1 \\ 2_1 & 1 & 1 \\ 2_3 & 1 & 1 \\ 4_1 & 4 & 1 \\ \hline  &&
\end{array} \qquad
\begin{array}{ccc} \hline \multicolumn{3}{c}{F_4,e=3}\\\hline
4_2 & 1 & 4 \\ 8_1 & 3 & 4 \\ 8_3 & 3 & 4 \\ 16_1 & 4 & 4 \\ \hline &&
\end{array}\qquad
\begin{array}{ccc} \hline \multicolumn{3}{c}{F_4,e=4}\\\hline 
1_1 & 0 & 1 \\ 4_2 & 1 & 4 \\ 9_1 & 2 & 4 \\ 6_1 & 4 & 1 \\ 12_1 & 4 & 4 \\ 
\hline \end{array} \qquad
\begin{array}{ccc}\hline  \multicolumn{3}{c}{F_4,e=6}\\\hline 
1_0 & 0 & 1 \\ 2_1 & 1 & 2 \\ 2_3 & 1 & 2 \\ 8_1 & 3 & 5 \\ 8_3 & 3 & 5 \\
\hline \end{array}\\ \\ \\
\begin{array}{ccc}\hline \multicolumn{3}{c}{E_6,e=2}\\\hline 
1_p & 0 & 1\\
6_p & 1 & 6\\
20_p & 2 & 14\\
15_q & 3 & 14\\
30_p & 3 & 10\\
60_p & 5 & 46\\\hline&&\\&&\\&& \\&&\\&&\end{array}\qquad
\begin{array}{ccc} \hline\multicolumn{3}{c}{E_6,e=3}\\\hline 
1_p & 0 & 1\\
6_p & 1 & 5\\
20_p & 2 & 14\\
15_p & 3 & 10\\
15_q & 3 & 1\\
30_p & 3 & 25\\
64_p & 4 & 10\\
60_p & 5 & 5\\
60_s & 7 & 14\\
80_s & 7 & 25\\\hline&&  \end{array}\qquad
\begin{array}{ccc} \hline \multicolumn{3}{c}{E_6,e=4}\\\hline 
1_p & 0 & 1\\
6_p & 1 & 6\\
15_p & 3 & 15\\
15_q & 3 & 8\\
81_p & 6 & 60\\
10_s & 7 & 1\\
80_s & 7 & 6\\
90_s & 7 & 15\\\hline &&\\&&\\&& \end{array}\qquad
\begin{array}{ccc}\hline \multicolumn{3}{c}{E_6,e=6}\\\hline 
1_p & 0 & 1\\
6_p & 1 & 6\\
20_p & 2 & 13\\
15_q & 3 & 14\\
30_p & 3 & 11\\
60_p & 5 & 32\\
24_p & 6 & 11\\
60_s & 7 & 14\\
80_s & 7 & 13\\
60_p' & 11 & 1\\
30_p' & 15 & 6\\ \hline\end{array}\\ \\ \\
\begin{array}{ccc}\hline \multicolumn{3}{c}{E_7,e=2}\\\hline 
1_a & 0 & 1\\
7_a' & 1 & 6\\
27_a & 2 & 14\\
35_b & 3 & 14\\
105_a' & 4 & 78\\
189_b' & 5 & 56\\
315_a' & 7 & 126\\ \hline && \\ \hline
\multicolumn{3}{c}{E_7,e=2}\\\hline
56_a' & 3 & 56\\
120_a & 4 & 64\\
280_b & 7 & 216\\\hline &&\\&&\\&&\\&&\\&&\\&&\\&&\\&&\\&&\\&&
\end{array}\qquad
\begin{array}{ccc}\hline \multicolumn{3}{c}{E_7,e=4}\\\hline 
1_a & 0 & 1\\
56_a' & 3 & 56\\
105_b & 6 & 48\\
210_a & 6 & 154\\
189_a & 8 & 35\\
405_a & 8 & 147\\
70_a & 16 & 21\\
315_a & 16 & 120\\\hline && \\ 
\hline \multicolumn{3}{c}{E_7,e=4}\\\hline
7_a' & 1 & 7\\
15_a' & 4 & 8\\
105_a' & 4 & 105\\
189_c' & 7 & 84\\
280_b & 7 & 168\\
378_a' & 9 & 21\\
210_b' & 13 & 27\\ \hline &&\\&&\\&&\\&&\\&&\end{array}\qquad
\begin{array}{ccc}\hline \multicolumn{3}{c}{E_7,e=4}\\\hline
27_a & 2 & 27\\
21_a & 3 & 21\\
35_b & 3 & 8\\
216_a' & 8 & 168\\
210_b & 10 & 7\\
105_c & 12 & 84\\
378_a & 14 & 105\\\hline && \\
\hline \multicolumn{3}{c}{E_7,e=4}\\\hline
21_b' & 3 & 21\\
120_a & 4 & 120\\
189_b' & 5 & 48\\
35_a' & 7 & 35\\
70_a' & 7 & 1\\
315_a' & 7 & 147\\
336_a & 13 & 154\\
405_a' & 15 & 56\\ \hline &&\\&&\\&&\\&&\\&& 
\end{array}\qquad
\begin{array}{ccc}\hline \multicolumn{3}{c}{E_7,e=3}\\\hline 
1_a & 0 & 1\\
21_a & 3 & 21\\
35_b & 3 & 34\\
120_a & 4 & 98\\
105_b & 6 & 7\\
168_a & 6 & 35\\
210_a & 6 & 91\\
280_b & 7 & 14\\
210_b & 10 & 49\\
420_a & 10 & 196\\\hline && \\ 
\hline \multicolumn{3}{c}{E_7,e=3}\\\hline
7_a' & 1 & 7\\
21_b' & 3 & 14\\
56_a' & 3 & 49\\
15_a' & 4 & 1\\
105_a' & 4 & 35\\
70_a' & 7 & 21\\
280_a' & 7 & 196\\
336_a' & 10 & 91\\
512_a & 11 & 98\\
84_a' & 13 & 34\\ 
\hline \end{array}\qquad
\begin{array}{ccc}\hline \multicolumn{3}{c}{E_7,e=6}\\\hline 
1_a & 0 & 1\\
7_a' & 1 & 7\\
21_b' & 3 & 13\\
21_a & 3 & 21\\
35_b & 3 & 27\\
15_a' & 4 & 14\\
105_a' & 4 & 77\\
105_b & 6 & 43\\
168_a & 6 & 43\\
210_a & 6 & 92\\
70_a' & 7 & 42\\
280_a' & 7 & 90\\
315_a' & 7 & 13\\
84_a & 10 & 14\\
210_b & 10 & 27\\
420_a & 10 & 92\\
210_b' & 13 & 1\\
420_a' & 13 & 77\\
280_a & 16 & 21\\
315_a & 16 & 7\\
\hline &&\\&&\end{array}\end{array}\]
Each table corresponds to a block of defect $\geq 2$. The first column 
specifies the set $\Lambda_{\zeta_e}^\circ$, the second column contains 
$\ba_\mu$ and the third column contains $\dim L_{\zeta_e}^\mu$ for 
$\mu \in \Lambda_{\zeta_e}^\circ$. 
\end{table}

\begin{table}[htbp] \caption{The sets $\Lambda_{\zeta_e}^\circ$ 
for type $E_8$} \label{can8}

\[\renewcommand{\arraystretch}{1.1}\renewcommand{\arraycolsep}{1.7pt}
\begin{array}{ccc}\hline \multicolumn{3}{c}{E_8,e=2}\\\hline 
1_x & 0 & 1\\
8_z & 1 & 8\\
35_x & 2 & 27\\
84_x & 3 & 48\\
50_x & 4 & 42\\
210_x & 4 & 202\\
560_z & 5 & 246\\
700_x & 6 & 126\\
1400_z & 7 & 792\\
1050_x & 8 & 651\\
1400_x & 8 & 378\\
4200_x & 12 & 1863\\\hline  && \\ 
\hline \multicolumn{3}{c}{E_8,e=2}\\\hline
112_z & 3 & 112\\
160_z & 4 & 160\\
400_z & 6 & 288\\
1344_x & 7 & 1184\\
2240_x & 10 & 1056\\
3360_z & 12 & 2016\\\hline &&\\ &&\\
\hline \multicolumn{3}{c}{E_8,e=4}\\\hline 
1_x & 0 & 1\\
35_x & 2 & 34\\
112_z & 3 & 77\\
50_x & 4 & 16\\
210_x & 4 & 176\\
567_x & 6 & 280\\
400_z & 6 & 96\\
175_x & 8 & 1\\
350_x & 8 & 70\\
1050_x & 8 & 336\\
1575_x & 8 & 946\\
525_x & 12 & 168\\
3360_z & 12 & 1654\\
2800_z & 13 & 1302\\
2835_x & 14 & 34\\
6075_x & 14 & 280\\
3150_y & 16 & 77\\
4480_y & 16 & 176\\
5670_y & 16 & 946\\\hline &&\\&&\\&&\\&&
\end{array}\;\;
\begin{array}{ccc}\hline \multicolumn{3}{c}{E_8,e=4}\\\hline 
8_z & 1 & 8\\
560_z & 5 & 560\\
1344_x & 7 & 784\\
840_z & 10 & 56\\
1400_zz & 10 & 832\\
4536_z & 13 & 2360\\
4200_z' & 21 & 1008\\
2240_x' & 28 & 1400\\\hline  && \\
\hline \multicolumn{3}{c}{E_8,e=4}\\\hline
28_x & 3 & 28\\
160_z & 4 & 160\\
300_x & 6 & 300\\
972_x & 10 & 512\\
840_x & 12 & 28\\
700_xx & 13 & 512\\
1344_w & 16 & 160\\
840_x' & 24 & 300\\\hline && \\
\hline \multicolumn{3}{c}{E_8,e=4}\\\hline
56_z & 7 & 56\\
1008_z & 7 & 1008\\
1400_z & 7 & 1400\\
3240_z & 9 & 832\\
2240_x & 10 & 8\\
4200_z & 15 & 2360\\
3200_x' & 21 & 784\\
4536_z' & 23 & 560\\\hline && \\
\hline \multicolumn{3}{c}{E_8,e=4}\\\hline
84_x & 3 & 84\\
700_x & 6 & 616\\
2268_x & 10 & 1652\\
4200_x & 12 & 1848\\
2100_x & 13 & 448\\
2016_w & 16 & 84\\
5600_w & 16 & 1652\\
4200_x' & 24 & 616\\
\hline &&\\&&\\&&\\&&\\&&\\&&\\&&\\&&\end{array}\;\;
\begin{array}{ccc}\hline \multicolumn{3}{c}{E_8,e=3}\\\hline 
1_x & 0 & 1\\
35_x & 2 & 35\\
28_x & 3 & 28\\
84_x & 3 & 48\\
50_x & 4 & 1\\
210_x & 4 & 147\\
300_x & 6 & 70\\
700_x & 6 & 518\\
1344_x & 7 & 497\\
175_x & 8 & 28\\
350_x & 8 & 322\\
1050_x & 8 & 35\\
1400_x & 8 & 1225\\
2240_x & 10 & 322\\
4096_z & 11 & 1036\\
4200_x & 12 & 147\\
700_xx & 13 & 48\\
3200_x & 15 & 497\\
4200_y & 16 & 518\\
4480_y & 16 & 1225\\\hline && \\ 
\hline \multicolumn{3}{c}{E_8,e=3}\\\hline
8_z & 1 & 8\\
112_z & 3 & 104\\
160_z & 4 & 56\\
560_z & 5 & 384\\
400_z & 6 & 8\\
448_z & 7 & 56\\
1400_z & 7 & 848\\
840_z & 10 & 448\\
1400_zz & 10 & 104\\
4096_x & 11 & 1896\\
4200_z & 15 & 384\\
5600_z & 15 & 1896\\
7168_w & 16 & 848\\ \hline 
&&\\&&\\&&\\&&\\&&\\&&\\&&\\&&\\&&\\&&\\&&\end{array}\;\;
\begin{array}{ccc}\hline \multicolumn{3}{c}{E_8,e=6}\\\hline 
1_x & 0 & 1\\
8_z & 1 & 8\\
35_x & 2 & 35\\
28_x & 3 & 28\\
84_x & 3 & 40\\
50_x & 4 & 41\\
210_x & 4 & 210\\
560_z & 5 & 279\\
300_x & 6 & 225\\
700_x & 6 & 86\\
56_z & 7 & 56\\
448_z & 7 & 85\\
1400_z & 7 & 489\\
175_x & 8 & 85\\
350_x & 8 & 266\\
1050_x & 8 & 660\\
1400_x & 8 & 259\\
840_z & 10 & 259\\
1400_zz & 10 & 40\\
840_x & 12 & 41\\
4200_x & 12 & 1906\\
2100_x & 13 & 1036\\
2400_z & 15 & 266\\
4200_z & 15 & 279\\
5600_z & 15 & 489\\
420_y & 16 & 1\\
1680_y & 16 & 56\\
4200_y & 16 & 660\\
4480_y & 16 & 8\\
4536_y & 16 & 225\\
5670_y & 16 & 28\\
4200_x' & 24 & 35\\
1400_x' & 32 & 210\\ \hline && \\
\hline \multicolumn{3}{c}{E_8,e=6}\\\hline
112_z & 3 & 112\\
160_z & 4 & 160\\
400_z & 6 & 288\\
1344_x & 7 & 1072\\
2240_x & 10 & 768\\
3360_z & 12 & 2128\\
3200_x & 15 & 2128\\
1344_w & 16 & 288\\
7168_w & 16 & 1072\\
3360_z' & 24 & 160\\
2240_x' & 28 & 112\\
\hline\end{array}\;\;
\begin{array}{ccc}\hline \multicolumn{3}{c}{E_8,e=5}\\\hline 
1_x & 0 & 1\\
28_x & 3 & 28\\
84_x & 3 & 83\\
567_x & 6 & 539\\
1344_x & 7 & 722\\
972_x & 10 & 166\\
2268_x & 10 & 1729\\
4096_z & 11 & 1078\\
168_y & 16 & 1\\
1134_y & 16 & 28\\
2688_y & 16 & 722\\
4536_y & 16 & 1729\\
4096_z' & 26 & 539\\
972_x' & 30 & 83\\ \hline && \\ && \\\hline 
\multicolumn{3}{c}{E_8,e=10}\\\hline 
1_x & 0 & 1\\
8_z & 1 & 8\\
28_x & 3 & 28\\
84_x & 3 & 75\\
567_x & 6 & 531\\
448_z & 7 & 372\\
1008_z & 7 & 449\\
1400_z & 7 & 786\\
972_x & 10 & 897\\
2268_x & 10 & 502\\
4536_z & 13 & 2406\\
1400_y & 16 & 449\\
3150_y & 16 & 372\\
4200_y & 16 & 897\\
4480_y & 16 & 786\\
4536_z' & 23 & 75\\
2268_x' & 30 & 531\\
448_z' & 37 & 1\\
1008_z' & 37 & 28\\
1400_z' & 37 & 8\\
\hline &&\\&&\\&&\\&&\\&&\\&&\\&&\\&&\\&&\end{array}\;\;
\begin{array}{ccc}\hline \multicolumn{3}{c}{E_8,e=8}\\\hline 
1_x & 0 & 1\\
35_x & 2 & 34\\
160_z & 4 & 160\\
567_x & 6 & 373\\
175_x & 8 & 174\\
1400_x & 8 & 992\\
1575_x & 8 & 1042\\
525_x & 12 & 152\\
2835_x & 14 & 1668\\
6075_x & 14 & 3516\\
2016_w & 16 & 174\\
5600_w & 16 & 1042\\
7168_w & 16 & 992\\
2835_x' & 22 & 1\\
6075_x' & 22 & 373\\
1400_x' & 32 & 34\\
1575_x' & 32 & 160\\\hline && \\ &&\\
\hline \multicolumn{3}{c}{E_8,e=12}\\\hline 
1_x & 0 & 1\\
35_x & 2 & 35\\
112_z & 3 & 76\\
50_x & 4 & 50\\
210_x & 4 & 99\\
400_z & 6 & 349\\
1050_x & 8 & 651\\
1400_x & 8 & 974\\
525_x & 12 & 449\\
3360_z & 12 & 1386\\
2800_z & 13 & 1202\\
1400_y & 16 & 99\\
2688_y & 16 & 651\\
4536_y & 16 & 974\\
2100_y & 20 & 449\\
3360_z' & 24 & 349\\
2800_z' & 25 & 76\\
1050_x' & 32 & 50\\
1400_x' & 32 & 1\\
210_x' & 52 & 35\\
\hline &&\\&&\\&&\\&&\\&&\\&&\end{array}\]
\end{table}

We can now state the main result of this article and its sequel 
\cite{gemu2}.

\medskip
\begin{thm} \label{mainthm} Recall our standing assumption that $e<\infty$ 
and $\operatorname{char} (k)=\ell>0$ where $\ell$ is a good prime for $W$.
Assume now that $W$ is of exceptional type and that $e\ell$ does not divide 
any degree of $W$. Then James's conjecture holds for $\cH$. More precisely, 
let $\Lambda_1$ be a $\xi$-block of $\Lambda$. By Proposition~\ref{block}, 
$\Lambda_1$ has a well-defined $\Phi_e$-defect, $\delta$ say. 
\begin{description}
\item[(a)] If $\delta=0$, then $\Lambda_1=\{\lambda\}$ is a singleton set;
we have $W_\xi^\lambda=L_\xi^\lambda$ and $W_{\zeta_e}^\lambda=
L_{\zeta_e}^\lambda$.
\item[(b)] If $\delta=1$, then the following hold:
\begin{description}
\item[(i)] We have $\ba_{\lambda} \neq \ba_{\lambda'}$ for any $\lambda \neq
\lambda'$ in $\Lambda_1$. Thus, we have a unique labelling $\Lambda_1=
\{\lambda_1, \lambda_2, \ldots, \lambda_n\}$ such that $\ba_{\lambda_1}<
\ba_{\lambda_2} < \cdots < \ba_{\lambda_n}$.
\item[(ii)] With the labelling in (i), we have $\Lambda_{1,\xi}^\circ=
\{\lambda_1,\ldots, \lambda_{n-1}\}$ and
\[ (W^{\lambda_i}_{\xi}:L^{\lambda_j}_{\xi})=(W^{\lambda_i}_{\zeta_e}: 
L^{\lambda_j}_{\zeta_e})=\left\{ \begin{array}{cl} 1 & \qquad \mbox{if 
$i=j$ or $i=j+1$},\\ 0 & \qquad\mbox{otherwise}.\end{array}\right.\]
\end{description}
\item[(c)] If $\delta \geq 2$, then $\Lambda_{1,\xi}^\circ$ and $\dim 
L^\mu_{\xi}$ for $\mu \in \Lambda_{1,\xi}^\circ$ are given by 
Tables~\ref{canf4} and~\ref{can8}.
\end{description}
\end{thm}

\medskip
\begin{rem} \label{2nd} The $\zeta_e$-blocks (together with their defect) 
of Iwahori--Hecke algebras of exceptional type are explicitly described 
in \cite[App.~F]{ourbuch}. We have verified all the statements of 
Theorem~\ref{mainthm}  using an actual implementation of the algorithms 
presented in Section~4, and their refinements in \cite{gemu2}. Some of 
these statements are known to hold by theoretical arguments. More precisely:
\begin{itemize}
\item[$\bullet$] The statement in (a)  follows from a general result about 
blocks of defect $0$ in symmetric algebras; see \cite[7.5.11]{ourbuch}.
\item[$\bullet$] The statement about $D_{\zeta_e,1}$ in (b) is proved, 
using general arguments, by a combination of \cite[\S 10]{mybaum}, 
\cite[\S 4]{mykl}, \cite[4.4]{GeRo2}. In \cite[\S 10]{mybaum} it is also
shown that these statements apply to $D_\xi$, if $\ell$ does not divide the
order of $W$.
\end{itemize}
Note also that, once James' Conjecture is established (in the form of
Corollary~\ref{Jconj1}), the complete decomposition matrices can be easily
determined: it is sufficient to compute them for {\em one} specialisation
$\theta \colon A \rightarrow k$ where $\mbox{char}(k)=\ell$ is a good prime
and $e\ell$ does not divide any degree of $W$. For the types $F_4$, $E_6$, 
$E_7$, these matrices were known before and can be found in \cite{GeLu}, 
\cite{mye6}, \cite{myhab}, \cite{muell}; for type $E_8$, see \cite{gemu2}.
\end{rem}

\section{Constructing the invariant bilinear form} \label{sec4}

We have seen in Proposition~\ref{largeell} that James' conjecture can be
verified once we have constructed the Gram matrices of the invariant 
bilinear forms on the cell modules $W^\lambda$. If $\cH$ is not too large,
we could actually do this by explicitly working out a cellular basis as in 
\cite[Exp.~4.3]{mycell} (type $B_2$) or Example~\ref{expG2} (type $G_2$). 
Using computers, it would also be possible to carry out similar computations 
in type $F_4$ and, perhaps, type $E_6$. However, this becomes totally 
unfeasible for type $E_7$ or $E_8$, where we do have to explore alternative
routes. The purpose of this section is to show how this can be done. 
Eventually, we will have to rely on computer calculations, but our aim is 
to develop a conceptual reduction of our problem where, at the end, 
standard programs like Parker's {\sf MeatAxe} \cite{Parker} and its 
variations can be applied. (See also Ringe's package \cite{Ri} which comes 
with extensive documentation and a variety of additions to Parker's 
original programs.) 

We keep the general setting of the previous section. Recall that $\cH$ is
defined over the ring $A=R[v,v^{-1}]$ where $R\subseteq \Q$ consists of
all fractions $a/b\in \Q$ such that $a \in \Z$ and $0 \neq b\in \Z$ is 
divisible by bad primes only. Let $K$ be the field of fractions of $A$. If 
$M$ is any $A$-module, we denote $M_K:=K \otimes_A M$.

Let $e \geq 2$ and $\theta \colon A \rightarrow k$ a ring homomorphism into 
a field $k$; let $\xi=\theta(u)\in k$. As before, if $M$ is any $A$-module, 
we denote $M_\xi:=k \otimes_A M$ where $k$ is regarded as an $A$-module via 
$\theta$. We say that $\theta$ is {\em $e$-regular} if $\mbox{char}(k)=
\ell>0$ is a good prime and $e\ell$ does not divide any degree of $W$. 
(These are precisely the conditions appearing in James' conjecture.) We 
will address the following three major issues which are sufficient for 
verifying that James' conjecture holds for a given algebra $\cH$:

\medskip
\begin{prob} \label{majprob} Let $e \geq 2$ be an integer which divides 
some degree of $W$. 
\begin{description}
\item[(a)] For any $\lambda \in \Lambda$, construct an explicit model for 
$W^\lambda$, that is, an $\cH$-module $V^\lambda$ which is free of finite 
rank over $A$ such that $V^\lambda_K\cong W^\lambda_K$. Determine 
$\Lambda_{\xi}^\circ$ and the decomposition matrix $D_{\xi}$ for at least
one $e$-regular specialisation $\theta \colon A \rightarrow k$.
\item[(b)] Show that, for each $\lambda \in \Lambda_{\zeta_e}^\circ$, the 
model $V^\lambda$ in (a) has the property that $V^\lambda_\xi\cong
W^\lambda_\xi$ for any $e$-regular specialisation $\theta\colon 
A\rightarrow k$.
\item[(c)] For any $\lambda \in \Lambda_{\zeta_e}^\circ$, determine the 
Gram matrix $Q^\lambda$  of an invariant bilinear form on $V^\lambda$ and 
show that $\text{rank}(Q^\lambda_\xi)=\mbox{rank}(G^\lambda_\xi)$ for 
any $e$-regular specialisation $\theta\colon A \rightarrow k$. 
\end{description}
Finally, compute $\text{rank}(Q^\lambda_{\zeta_e})$ and find the {\em finite}
set of prime numbers $\mathcal{P}_e$ such that 
\[ \text{rank}(Q^\lambda_\xi)=\text{rank}(Q_{\zeta_e}^\lambda)
\qquad \mbox{if $\ell \not\in \mathcal{P}_e$}.\]
\end{prob}

\medskip
\subsection{Solving Problem~\ref{majprob}(a)} \label{subP1}

Natural candidates for models for the cell representations of $\cH$ are the 
representations afforded by $W$-graphs. In fact, Gyoja \cite{Gy1} has 
shown that every irreducible representation of $\cH_K$ is afforded by a 
$W$-graph.  We recall:

\medskip
\begin{defn}[Kazhdan--Lusztig \protect{\cite{KaLu}}] \label{wgraph} A
$W$-graph for $\cH$ consists of the following data:
\begin{description}
\item[(a)] a set $X$ together with a map $I$ which assigns to
each $x \in X$ a set $I(x) \subseteq S$;
\item[(b)] a collection of elements $\mu_{x,y}\in \Z$, where
$x,y \in X$, $x\neq y$.
\end{description}
These data are  subject to the following requirements. Let $V$ be a free 
$A$-module with a basis $\{e_y \mid y \in X\}$. For each $s \in S$, define 
an $A$-linear map $\sigma_s \colon V \rightarrow V$ by
\begin{alignat*}{2}
\sigma_s(e_y)&=v^{2}e_y+\sum_{\atop{x \in X}{s\in I(x)}} v\mu_{x,y} e_{x} 
&&\quad\mbox{if $s\not\in I(y)$},\\
\sigma_s(e_y)&=-e_y &&\quad \mbox{if $s \in I(y)$}.\end{alignat*}
Then we require that the assignment $T_s \mapsto \sigma_s$ defines a
representation of $\cH$.
\end{defn}

\medskip
Thus, in a representation afforded by a $W$-graph, each generator $T_s$ ($s 
\in S$) of $\cH$ is represented by a matrix of a particularly simple form. 
Recently, Howlett and Yin \cite{How}, \cite{yin0} explicitly constructed 
$W$-graphs for all irreducible representations for Iwahori--Hecke algebras 
of type $E_7$, $E_8$. In combination with earlier results of Naruse 
\cite{Naruse0} on types $F_4$ and $E_6$, we now have $W$-graphs for all 
irreducible representations of algebras of exceptional type. These 
$W$-graphs are electronically accessible through Michel's development 
version \cite{jmich} of the computer algebra system {\sf CHEVIE} \cite{chv}. 
Thus, we do have a collection of explicitly given $\cH$-modules
\[ \{V^\lambda \mid \lambda \in \Lambda\}\]
such that each $V^\lambda$ is free of finite rank over $A$ and $V^\lambda_K
\cong E_v^\lambda \cong W_K^\lambda$. 

Now let $\theta \colon A \rightarrow k$ be an $e$-regular  specialisation.
Using the {\tt CHOP} function in Ringe's version \cite{Ri} of the 
{\sf MeatAxe}, we can decompose each $V_\xi^\lambda$ into its irreducible 
constituents. Thus, we obtain:

\smallskip
\begin{itemize}
\item[$\bullet$] $\Irr(\cH_k(W,\xi))=\{M_1,\ldots,M_r\}$ and 
\item[$\bullet$] the decomposition numbers $(V_{\xi}^\lambda:M_i)$ for
$\lambda \in \Lambda$ and $1\leq i\leq r$.
\end{itemize} 
\smallskip
Note that, by Remark~\ref{brauer}, we have $(W^\lambda_\xi:M_i)=
(V_\xi^\lambda:M_i)$ for all $\lambda \in \Lambda$ and $1\leq i \leq r$.
The relations ($\Delta$) in Section~\ref{sec2} immediately imply the
following ``identification result'':

\medskip
\begin{lem} \label{identify} Let $i \in \{1,\ldots,r\}$. Then the unique 
$\mu \in \Lambda_{\xi}^\circ$ such that $M_i=L_\xi^\mu$ is determined 
by the conditions that $(W_{\xi}^\mu:M_i)=1$ and 
\[ \ba_\mu\leq \ba_\lambda \quad \mbox{for all $\lambda \in \Lambda$ such 
that $(W_\xi^\lambda:M_i)\neq 0$}.\]
\end{lem}

By Theorem~\ref{thm13} and Lemma~\ref{factor}, we have $\Lambda_\xi^\circ=
\Lambda_{\zeta_e}^\circ$. Thus, we are able to determine the sets
$\Lambda^\circ_{\zeta_e}$ for any $e \geq 2$. This already yields the
information contained in the first columns in Table~\ref{canf4} 
and~\ref{can8}.

\medskip
\subsection{Solving Problem~\ref{majprob}(b)} \label{subP2}

Let us fix $e \geq 2$ and an element $\lambda \in \Lambda_{\zeta_e}^\circ$.
As discussed above, we have an $\cH$-module $V^\lambda$ such that 
$W^\lambda_K \cong V^\lambda_K$. Now let $\theta\colon A \rightarrow k$ be 
any $e$-regular specialisation. In general, without any further knowledge
about $V^\lambda$, we cannot expect that we also have $W^\lambda_\xi 
\cong V^\lambda_\xi$. The following result gives  a precise condition for 
when this is the case.

\medskip
\begin{prop} \label{help1} Assume that there exists some $e$-regular
specialisation $\theta_0 \colon A \rightarrow k_0$ such that 
$V^\lambda_{\xi_0}$ (where $\xi_0=\theta_0(u)$) has a unique maximal 
submodule $U^\lambda$, and we have $V^\lambda_{\xi_0}/U^\lambda \cong 
L^\lambda_{\xi_0}$. Then $V^\lambda_{\zeta_e} \cong W^\lambda_{\zeta_e}$
and $V^\lambda_{\xi} \cong W^\lambda_{\xi}$ for any $e$-regular 
specialisation $\theta\colon A \rightarrow k$.
\end{prop}

\begin{proof} The module $W^\lambda$ has a standard basis $\{C_{\fs} \mid
\fs \in M(\lambda)\}$;  let $\rho^\lambda \colon \cH \rightarrow 
M_{d_\lambda}(A)$ be the corresponding matrix representation. The module 
$V^\lambda$ also has a standard basis, arising from the underlying 
$W$-graph; let $\sigma^\lambda \colon \cH \rightarrow M_{d_\lambda}(A)$ be 
the corresponding matrix representation. Since $V^\lambda_K \cong 
W^\lambda_K$, there exists an invertible matrix $P^\lambda \in 
M_{d_\lambda}(K)$ such that 
\[\rho^\lambda(T_w)P^\lambda=P^\lambda\sigma^\lambda(T_w)\qquad\mbox{for 
all $w \in W$}.\]
Multiplying $P^\lambda$ by a suitable scalar, we may assume without loss of 
generality that 
\begin{itemize}
\item[$\bullet$] all entries of $P^\lambda$ lie in ${\Z}[v]$ and 
\item[$\bullet$] the greatest common divisor of all non-zero entries of 
$P^\lambda$ is $1$.
\end{itemize}
(Here we use the fact that $R$ was chosen to be contained in $\Q$.) These
conditions uniquely determine $P^\lambda$ up to a sign. Let $\delta:=
\det(P^\lambda)\neq 0$. We need to obtain some more precise information
about the irreducible factors of $\delta$. Let us write
\[ \delta=mf_1f_2\cdots f_r \qquad \mbox{where $0 \neq m \in \Z$ and
$f_1,\ldots,f_r \in {\Z}[v]\setminus \Z$ are irreducible}.\]
First we claim that $m$ is divisible by bad primes only. Indeed, let
$p$ be a prime number which is good for $W$. Then $p$ generates a prime 
ideal in $R$; let $F=\F_p(v)$.  We obtain a specialisation $\alpha 
\colon A \rightarrow F$ by reducing the coefficients of polynomials in 
$A$ modulo~$p$. We have a corresponding specialised algebra $\cH_F(W,u)$. 
Since $\alpha(P_W)\neq 0$, we conclude that $\cH_F(W,u)$ is semisimple and 
the specialized cell modules $W^\lambda_F$ are all irreducible; see  
Proposition~\ref{prop11}. Now note that not all entries of $P^\lambda$ are
divisible by $p$. Hence, reducing the entries of $P^\lambda$ modulo $p$, 
we obtain a non-zero matrix defining a non-trivial module homomorphism
$V^\lambda_F \rightarrow W^\lambda_F$. Since $W^\lambda_F$ is irreducible 
and $\dim W^\lambda_F=\dim V^\lambda_F$, this homomorphism must be an 
isomorphism. Consequently, $P^\lambda$ is invertible modulo $p$ and so 
$p$ cannot divide $m$.

A similar argument shows that each $f_i$ divides $P_W(v^2)$. Indeed, assume 
that $f\in {\Z}[v]$ is a non-constant irreducible polynomial which does not
divide $P_W(v^2)$. Then we have a canonical ring homomorphism $\beta \colon 
A \rightarrow F$ where $F={\Q}[v]/(f)$. Again, the corresponding 
specialised algebra $\cH_F(W,\theta(u))$ is semisimple since $\beta
(P_W)\neq 0$. Arguing as above, we conclude that $f$ does not divide 
$\det(P^\lambda)$. Thus, each $f_i$ must divide $P_W(v^2)$.

Now consider the specialisation $\theta_e \colon A \rightarrow \C$ which
sends $v$ to $\zeta_{2e}$. We can actually regard this as a map with
image in ${\Q}(\zeta_{2e})$ and work with ${\Q}(\zeta_{2e})$ instead of $\C$ 
as base field. Thus, $W_{\zeta_e}^\lambda$ and $V^\lambda_{\zeta_e}$ can be
regarded as ${\Q}(\zeta_{2e})$-vectorspaces and modules for the specialised 
algebra $\cH_{\Q(\zeta_{2e})}(W,\zeta_e)$. Let $\cO$ be a discrete valuation 
ring as in Remark~\ref{compare} with respect to the specialisation 
$\theta_0$; we have a corresponding decomposition map
\[ d_{\xi_0}^e \colon R_0(\cH_{\Q(\zeta_{2e})}(W,\zeta_e))\rightarrow 
R_0(\cH_{k_0}(W,\xi_0)).\] 
Once again, since the greatest common divisor of all its entries is $1$, 
the matrix $P^\lambda$ induces a non-trivial module homomorphism 
$V^\lambda_{\zeta_e} \rightarrow W^\lambda_{\zeta_e}$. We claim that this 
also is an isomorphism. To prove this, let $M \subseteq V^\lambda_{\zeta_e}$ 
be the kernel of the map $V^\lambda_{\zeta_e} \rightarrow 
W^\lambda_{\zeta_e}$; then $M$ is a proper submodule of 
$V^\lambda_{\zeta_e}$. By a standard result (see \cite[23.7]{CR1}), there 
exists a proper submodule $N \subseteq V_{\xi_0}^\lambda$ such that 
\[ d_{\xi_0}^e([M])=[N] \qquad \mbox{and} \qquad d_{\xi_0}
([V^\lambda_{\zeta_e}/M])=[V^\lambda_{\xi_0}/N].\]
If $L_{\zeta_e}^\lambda$ were a composition factor of $M$, then 
$L_{\xi_0}^\lambda$ would be a composition factor of $N$ by 
Lemma~\ref{factor}(b). But then, by our assumption on $V_{\xi_0}^\lambda$ 
and since $N\subseteq U$, the simple module $L_{\xi_0}^\lambda$ 
would appear at least twice as a composition factor of $V_{\xi_0}^\lambda$,
which is absurd. So we conclude that $L_{\zeta_e}^\lambda$ is not a 
composition factor of $M$. Hence, $L_{\zeta_e}^\lambda$ will be a 
composition factor  of the image of the map $V^\lambda_{\zeta_e} 
\rightarrow W^\lambda_{\zeta_e}$. But, by \cite[Prop.~3.2]{GrLe}, 
$L_{\zeta_e}^\lambda$ is a simple quotient of $W^{\lambda}_{\zeta_e}$,
the kernel of the canonical map $W^\lambda_{\zeta_e}\rightarrow 
L^\lambda_{\zeta_e}$ is the unique maximal submodule of 
$W^\lambda_{\zeta_e}$, and $L^\lambda_{\zeta_e}$ is not a composition 
factor of that kernel. So we conclude that the map $V^\lambda_{\zeta_e}
\rightarrow W^\lambda_{\zeta_e}$ is surjective and, hence, an isomorphism. 
It follows that $\delta$ is not divisible by $\Phi_{2e}(v)$.  If $e$ is odd,
we can also consider the specialisation $\theta_e' \colon A \rightarrow \C$ 
sending $v$ to $\zeta_e^{(e+1)/2}$ (the other square root of $\zeta_e$, 
which is a root of $\Phi_e(v)$). Then a similar argument shows that 
$\delta$ is not divisible by $\Phi_e(v)$. Thus, we have reached the 
following conclusions:
\begin{itemize}
\item[$\bullet$] $m$ is divisible by bad primes only;
\item[$\bullet$] each $f_i$ divides $P_W(v^2)$;
\item[$\bullet$] each $f_i$ is coprime to $\Phi_e(v^2)$.
\end{itemize}

We can now complete the proof as follows. Let $\theta \colon A \rightarrow 
k$ be any $e$-regular specialisation. Assume that $\theta(\delta)=0$. Since 
the characteristic of $k$ is a good prime, we must have $\theta(f_i)=0$ for 
some $i\in\{1,\ldots,r\}$. Since each $f_i$ divides $P_W(v^2)$, there 
exists some $d \geq 2$ such that $\Phi_d(v^2)$ divides $P_W(v^2)$ and 
$f_i$ divides $\Phi_d(v^2)$. By Remark~\ref{signi}, we conclude that $d=e$. 
Thus, we see that $f_i$ divides $\Phi_{e}(v^2)$, a contradiction. Hence, 
our assumption was wrong and so we do have $\theta(\delta)\neq 0$. Thus, 
we have shown that $P^\lambda$ induces an isomorphism $V^\lambda_\xi 
\stackrel{\sim}{\rightarrow} W^\lambda_\xi$.
\end{proof} 

Let $\theta_0 \colon A \rightarrow k_0$ be a specialisation as in
Proposition~\ref{help1}. Using the {\tt MKSUB} function in Ringe's version 
\cite{Ri} of the {\sf MeatAxe} (see also \cite{LuMuRi}), we can determine 
the complete submodule lattice of $V^\lambda_{\xi_0}$. Using the {\tt CHOP} 
function and Lemma~\ref{identify} as discussed in the previous subsection, 
we can identify the various irreducible constituents of $V^\lambda_{\xi_0}$ 
and check that the assumption of Proposition~\ref{help1} is satisfied. Thus, 
Problem~\ref{majprob}(b) is solved.

It might actually be true that $W^\lambda$ and $V^\lambda$ are 
isomorphic as $\cH$-modules, but we have not been able to prove this. 
We would like to state this as a conjecture:

\medskip
\begin{conj} \label{Cwgraph} Assume that, for each $\lambda \in \Lambda$, 
we are given a $W$-graph affording an $\cH$-module $V^\lambda$ such that 
$V^\lambda_K \cong E_v^\lambda$. Then the cellular basis in 
Theorem~\ref{Hcellbase} can be chosen such that $W^\lambda \cong V^\lambda$
for all $\lambda \in \Lambda$.
\end{conj}

\medskip
\subsection{Solving Problem~\ref{majprob}(c)} \label{subP3}

Let $\lambda \in \lambda_{\zeta_e}^\circ$ and $G^\lambda$ be the Gram 
matrix of the invariant bilinear form $\phi^\lambda$ with respect to the 
standard basis of $W^\lambda$. Instead of $W^\lambda$, we now consider 
the module $V^\lambda$ and assume that the hypotheses of 
Proposition~\ref{help1} are satisfied. Thus, we have $V^\lambda_{\zeta_e} 
\cong W^\lambda_{\zeta_e}$ and $V^\lambda_\xi \cong W^\lambda_\xi$ for 
any $e$-regular specialisation $\theta \colon A \rightarrow k$.

Let $\sigma^\lambda \colon \cH \rightarrow M_{d_\lambda}(A)$ be the matrix 
representation afforded by $V^\lambda$ with respect to the standard basis
arising from the underlying $W$-graph. Our task now is to find some 
non-zero matrix $Q^\lambda \in M_{d_\lambda}(A)$ such that 
\begin{equation*}
 Q^\lambda \cdot \sigma^\lambda(T_s)=\sigma^\lambda(T_s)^{\text{tr}}
\cdot Q^\lambda \qquad \mbox{for all $s \in S$}.\tag{$*$}
\end{equation*}
Note that ($*$) implies that $Q^\lambda \cdot \sigma^\lambda(T_{w^{-1}})=
\sigma^\lambda(T_w)^{\text{tr}} \cdot Q^\lambda$ for all $w \in W$. So 
any solution to ($*$) is the Gram matrix of an invariant bilinear form on 
$V^\lambda$. Multiplying $Q^\lambda$ by a suitable scalar, we may assume 
without loss of generality that
\begin{itemize}
\item[$\bullet$] all entries of $Q^\lambda$ lie in ${\Z}[v]$ and
\item[$\bullet$] the greatest common divisor of all non-zero entries of
$Q^\lambda$ is $1$.
\end{itemize}
Note that, by Schur's Lemma, any two matrices satisfying ($*$) are scalar 
multiples of each other. Hence, the above conditions uniquely determine
$Q^\lambda$ up to a sign.

\medskip
\begin{lem} \label{help2} Assume that $Q^\lambda$ is a solution to 
{\rm ($*$)} satisfying the above conditions. Then 
\[\operatorname{rank}(Q_{\zeta_e})=\operatorname{rank}(G^\lambda_{\zeta_e})
\qquad \mbox{and}\qquad \operatorname{rank}(Q_{\xi})=\operatorname{rank}
(G^\lambda_{\xi})\]
for any $e$-regular specialisation $\theta \colon A \rightarrow k$.
\end{lem}

\begin{proof} We are assuming that $\lambda \in \Lambda_{\zeta_e}^\circ=
\Lambda_\xi^\circ$, so we have $G^\lambda_{\zeta_e}\neq 0$ and 
$G^\lambda_{\xi}\neq 0$.

Now let $P^\lambda$ be as in the proof of Proposition~\ref{help1} and set
$\tilde{Q}^\lambda:=(P^\lambda)^{\text{tr}}G^\lambda P^\lambda$. Then
$\tilde{Q}^\lambda$ is a solution to ($*$) and so there exists some
$0 \neq \alpha \in K$ such that $\tilde{Q}^\lambda=\alpha Q^\lambda$.
Since all three matrices $G^\lambda$, $Q^\lambda$ and $\tilde{Q}^\lambda$
have all their entries in ${\Z}[v,v^{-1}]$ and since the greatest 
common divisior of the entries of $Q^\lambda$ is $1$, we can conclude 
that $\alpha \in {\Z}[v,v^{-1}]$. 

Now, in the proof of Proposition~\ref{help1}, we have actually seen that
$P^\lambda_{\zeta_e}$ and $P^\lambda_\xi$ are invertible. Since we also
have $G^\lambda_{\zeta_e}\neq 0$ and $G^\lambda_{\xi}\neq 0$, it follows that 
\[\operatorname{rank}(\tilde{Q}_{\zeta_e})=\operatorname{rank}
(G^\lambda_{\zeta_e})>0 \qquad \mbox{and}\qquad \operatorname{rank}
(\tilde{Q}_{\xi})=\operatorname{rank} (G^\lambda_{\xi})>0.\]
But then it also follows that $\theta_e(\alpha)\neq 0$ and 
$\theta(\alpha)\neq 0$. Hence, we have $\operatorname{rank}(
\tilde{Q}_{\zeta_e})=\operatorname{rank}(Q^\lambda_{\zeta_e})$ and
$\operatorname{rank}(\tilde{Q}_{\xi})=\operatorname{rank}(Q^\lambda_{\xi})$,
and this yields the desired statement.
\end{proof}

It remains to show how a solution to ($*$) can actually be computed. Note 
that ($*$) constitutes a system of $|S|d_\lambda^2$ linear equations for 
the $d_\lambda^2$ entries of $Q^\lambda$. If $d_\lambda$ is not too large, 
this can be solved directly. However, in type $E_8$, we have $d_\lambda=7168$
for some $\lambda$, and our system of linear equations simply becomes too 
large. In such cases, different techniques are required which are based 
on the following result: 

\medskip
\begin{thm}[Benson--Curtis; see \protect{\cite[\S 6.3]{ourbuch}}] 
\label{becu} Each $E^\mu\in \Irr(W)$ is of parabolic type, that is, there 
exists a subset $I \subseteq S$ such that the restriction of $E^\mu$ 
to the parabolic subgroup $W_I \subseteq W$ contains the trivial 
representation of $W_I$ just once. A similar statement holds when
``trivial representation'' is replaced by ``sign representation''.
\end{thm}

\medskip
Now the main idea is as follows: Since the bijection $\Irr(W) 
\leftrightarrow \Irr(\cH_K)$ arising from Tits' Deformation Theorem is 
compatible with restriction to parabolic subgroups and subalgebras (see 
\cite[9.1.9]{ourbuch}), the above result means that there exists a subset 
$I \subseteq S$ such that 
\[ \dim_K \Bigl(\bigcap_{s \in I} \ker\bigl(\sigma^\lambda_K(T_s+T_1)
\bigr)\Bigr)=1;\]
let $e_1 \in K^{d_\lambda}$ be a vector spanning this one-dimensional space.
Similarly, 
\[ \dim_K\Bigl(\bigcap_{s \in I} \ker\bigl(\sigma^\lambda_K(T_s+
T_1)^{\text{tr}} \bigr)\Bigr)=1;\]
let $v_1\in K^{d_\lambda}$ be a vector spanning this one-dimensional space.
Now, since $\sigma_K^\lambda$ is an irreducible representation of $\cH_K$, 
there exist $w_2,\ldots,w_{d_\lambda} \in W$ such that the vectors
\[ e_1,\quad e_2:=\sigma^\lambda(T_{w_2})e_1, \quad e_3:=\sigma^\lambda
(T_{w_3})e_1,\quad \ldots, \quad e_{d_\lambda}:=\sigma^\lambda
(T_{w_{d_\lambda}})e_1,\]
form a basis of $K^{d_\lambda}$. Then the vectors
\[ v_1,\quad v_2:=\sigma^\lambda(T_{w_2^{-1}})^{\text{tr}}v_1, \quad 
v_3:=\sigma^\lambda (T_{w_3^{-1}})^{\text{tr}}v_1,\quad \ldots, \quad 
v_{d_\lambda}:=\sigma^\lambda (T_{w_{d_\lambda}^{-1}})^{\text{tr}}v_1,\]
will also form a basis of $K^{d_\lambda}$. Hence, there exists a unique
invertible matrix $\tilde{Q}^\lambda\in M_{d_\lambda}(K)$ such that 
$v_i= \tilde{Q}^\lambda e_i$ for $1\leq i \leq d_\lambda$. Then 
$\tilde{Q}^\lambda \cdot \sigma^\lambda(T_{w}) \cdot
(\tilde{Q}^\lambda)^{-1}=\sigma^\lambda(T_{w^{-1}})^{\text{tr}}$ for 
all $w\in W$ and so $\tilde{Q}^\lambda$ is a solution to ($*$). 
Multiplying by a suitable scalar, we obtain $Q^\lambda$. 

The above technique is known as the ``standard base'' algorithm; see the 
{\tt ZSB} function of Ringe's {\sf MeatAxe} \cite{Ri} and its description. 
In practice, we did not apply it to $\sigma^\lambda$ itself but to various 
specialisations into finite fields such that the specialised algebra 
remains semisimple. For each such specialisation, we use the {\tt ZSB} 
function to find the Gram matrix of an invariant bilinear form. Using 
interpolation and modular techniques (Chinese Remainder), one can recover 
$Q^\lambda$ from these specialisations.

Having computed $Q^\lambda$, we substitute $v \mapsto \sqrt[2e]{1}$ and
determine the rank of the specialised matrix. Arguing as in the proof of
Proposition~\ref{largeell}, we find the {\em finite} set of prime numbers 
$\mathcal{P}_e$ such that $\text{rank}(Q^\lambda_\xi)=\text{rank}
(Q_{\zeta_e}^\lambda)$ if $\ell \not\in \mathcal{P}_e$. See \cite{gemu2}
for further details.

\medskip
\begin{rem} \label{zsb} Assume we are in the above setting, where $I \subseteq
S$ is a subset such that the restriction of $E^\lambda$ to $W_I$ contains
the sign representation exactly once. Then, by the formulas in 
Definition~\ref{wgraph}, the vector $e_1$ can be taken to be contained in 
the standard basis of $K^{d_\lambda}$. Since $v_1=\tilde{Q}^\lambda e_1$,
we conclude that $v_1$ is a column of the matrix $\tilde{Q}^\lambda$. In
other words, using Theorem~\ref{becu}, one column of the matrix 
$\tilde{Q}^\lambda$ can be computed by simply determining the intersection 
of the kernels of the maps $\sigma_K^\lambda (T_s+T_1)^{\text{tr}}$ where 
$s$ runs over the generators in $I$.
\end{rem}

\begin{table} 
\caption{$W$-graph and invariant bilinear form for the 
representation $10_s$ in type $E_6$}
\begin{center}
\begin{picture}(300,150)
\put( 60,140){\line(1,0){190}}
\put(250,140){\line(0,-1){59}}
\put(250, 10){\line(0,1){59}}
\put( 60,140){\line(0,-1){24}}
\put( 60, 10){\line(0,1){24}}
\put( 67,110){\line(1,0){146}}
\put( 60, 10){\line(1,0){190}}
\put( 60, 20){\line(1,0){160}}
\put( 67, 40){\line(1,0){146}}
\put( 30,130){\line(1,0){190}}
\put( 30, 20){\line(1,0){190}}
\put( 30, 20){\line(0,1){49}}
\put( 30,130){\line(0,-1){49}}
\put(220, 20){\line(0,1){14}}
\put(220,130){\line(0,-1){14}}
\put(220, 46){\line(0,1){23}}
\put(220,104){\line(0,-1){23}}
\put( 60, 46){\line(0,1){23}}
\put( 60,104){\line(0,-1){23}}
\put( 66,107){\line(4,-3){39}}
\put( 66, 44){\line(4,3){38}}
\put(214, 44){\line(-4,3){38}}
\put(214,106){\line(-4,-3){38}}
\bezier{200}(36,78)(70,100)(104,76)
\bezier{400}(67,73)(140,35)(206,73)
\bezier{200}(176,76)(210,100)(244,79)
\put( 30,75){\oval(14,12)}
\put( 24,73){$\scriptstyle{145}$}
\put( 59,110){\oval(16,12)}
\put( 52,108){$\scriptstyle{2146}$}
\put( 60,75){\oval(14,12)}
\put( 54,73){$\scriptstyle{146}$}
\put( 60, 40){\oval(14,12)}
\put( 54, 38){$\scriptstyle{346}$}
\put(110,75){\oval(12,12)}
\put(108,73){$\scriptstyle{4}$}
\put(170,75){\oval(13,12)}
\put(164,73){$\scriptstyle{235}$}
\put(217,75){\oval(22,12)}
\put(207,73){$\scriptstyle{12356}$}
\put(220, 40){\oval(14,12)}
\put(216, 38){$\scriptstyle{35}$}
\put(220,110){\oval(14,12)}
\put(215,108){$\scriptstyle{215}$}
\put(250,75){\oval(14,12)}
\put(244,73){$\scriptstyle{236}$}
\put( 37,75){\line(1,0){16}}
\put(116,75){\line(1,0){47}}
\put(228,75){\line(1,0){15}}
\end{picture}
\end{center}
\begin{align*}
Q^{10_s}&= \left[\begin{array}{cccccccccc}
  v^6 {+} 3v^4 {+} 3v^2 {+} 1&2v^4 {+} 2v^2&2v^4 {+} 2v^2&{-}v^5 {-} 2v^3 {-}
 v& 2v^4 {+} 2v^2 \\
   2v^4 {+} 2v^2&v^6 {+} 3v^4 {+} 3v^2 {+} 1&2v^4 {+} 2v^2&{-}v^5 {-} 2v^3 
{-} v& 2v^4 {+} 2v^2 \\
   2v^4 {+} 2v^2&2v^4 {+} 2v^2&v^6 {+} 3v^4 {+} 3v^2 {+} 1&{-}v^5 {-} 2v^3 
{-} v& 2v^4 {+} 2v^2 \\
   {-}v^5 {-} 2v^3 {-} v&{-}v^5 {-} 2v^3 {-} v&{-}v^5 {-} 2v^3 {-} v&
      v^6 {+} 2v^4 {+} 2v^2 {+} 1&{-}v^5 {-} 2v^3 {-} v \\
   2v^4 {+} 2v^2&2v^4 {+} 2v^2&2v^4 {+} 2v^2&{-}v^5 {-} 2v^3 {-} v&
      v^6 {+} 3v^4 {+} 3v^2 {+} 1 \\
   {-}v^5 {-} 2v^3 {-} v&{-}v^5 {-} 2v^3 {-} v&{-}v^5 {-} 2v^3 {-} v&v^4 
{+} v^2&{-}2v^3 \\
   2v^4 {+} 2v^2&2v^4 {+} 2v^2&2v^4 {+} 2v^2&{-}2v^3&2v^4 {+} 2v^2 \\
   {-}v^5 {-} 2v^3 {-} v&{-}v^5 {-} 2v^3 {-} v&{-}2v^3&v^4 {+} v^2&{-}v^5 
{-} 2v^3 {-} v \\
   {-}v^5 {-} 2v^3 {-} v&{-}2v^3&{-}v^5 {-} 2v^3 {-} v&v^4 {+} v^2&{-}v^5 
{-} 2v^3 {-} v \\
   {-}2v^3&{-}v^5 {-} 2v^3 {-} v&{-}v^5 {-} 2v^3 {-} v&v^4 {+} v^2&{-}v^5 
{-} 2v^3 {-} v \\
\end{array}\right.\\[2mm]
&\qquad \left.\begin{array}{cccccccccc}
 {-}v^5 {-} 2v^3 {-} v&2v^4 {+} 2v^2&{-}v^5 {-} 2v^3 {-} v&{-}v^5 {-} 2v^3 
{-} v& {-}2v^3 \\
   {-}v^5 {-} 2v^3 {-} v&2v^4 {+} 2v^2&{-}v^5 {-} 2v^3 {-} v&{-}2v^3&{-}v^5 
{-} 2v^ 3 {-} v \\
   {-}v^5 {-} 2v^3 {-} v&2v^4 {+} 2v^2&{-}2v^3&{-}v^5 {-} 2v^3 {-} v&
{-}v^5 {-} 2v^3 {-} v \\  
v^4 {+} v^2&{-}2v^3&v^4 {+} v^2&v^4 {+} v^2&v^4 {+} v^2 \\
   {-}2v^3&2v^4 {+} 2v^2&{-}v^5 {-} 2v^3 {-} v&{-}v^5 {-} 2v^3 {-} v&
      {-}v^5 {-} 2v^3 {-} v \\
   v^6 {+} 2v^4 {+} 2v^2 {+} 1&{-}v^5 {-} 2v^3 {-} v&v^4 {+} v^2&v^4 {+} 
v^2& v^4 {+} v^2 \\
   {-}v^5 {-} 2v^3 {-} v&v^6 {+} 3v^4 {+} 3v^2 {+} 1&{-}v^5 {-} 2v^3 {-} v&
      {-}v^5 {-} 2v^3 {-} v&{-}v^5 {-} 2v^3 {-} v \\
   v^4 {+} v^2&{-}v^5 {-} 2v^3 {-} v&v^6 {+} 2v^4 {+} 2v^2 {+} 1&v^4 {+} v^2&
      v^4 {+} v^2 \\
   v^4 {+} v^2&{-}v^5 {-} 2v^3 {-} v&v^4 {+} v^2&v^6 {+} 2v^4 {+} 2v^2 {+} 1&
      v^4 {+} v^2 \\
   v^4 {+} v^2&{-}v^5 {-} 2v^3 {-} v&v^4 {+} v^2&v^4 {+} v^2&v^6 {+} 2v^4 {+}
 2v^2 {+} 1 \end{array}\right]
\end{align*}
\label{wgraphf} 
\end{table}

\medskip
\begin{exmp} \label{dim10} In general, the matrix $Q^\lambda$ is far from
being sparse. We just give one example. Let $W$ be of type $E_6$ with 
Dynkin diagram
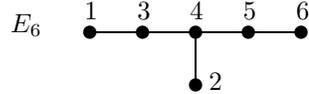
\begin{center} 
\begin{picture}(200,40)
\put(30, 25){$E_6$}
\put(60, 25){\circle*{5}}
\put(58, 30){$1$}
\put(60, 25){\line(1,0){20}}
\put(80, 25){\circle*{5}}
\put(78, 30){$3$}
\put(80, 25){\line(1,0){20}}
\put(100, 25){\circle*{5}}
\put(98, 30){$4$}
\put(100, 25){\line(0,-1){20}}
\put(100, 05){\circle*{5}}
\put(105, 03){$2$}
\put(100, 25){\line(1,0){20}}
\put(120, 25){\circle*{5}}
\put(118, 30){$5$}
\put(120, 25){\line(1,0){20}}
\put(140, 25){\circle*{5}}
\put(138, 30){$6$}
\end{picture}
\end{center}
Consider the unique $10$-dimensional irreducible representation, which is
denoted $10_s$ in \cite[Table~C.4]{ourbuch}. By Naruse \cite{Naruse0}, a 
$W$-graph is given by Table~\ref{wgraphf}. (The numbers inside a circle 
specify the subset $I(x)$; all $\mu_{x,y}$ are $0$ or $1$; we have an edge 
between $x$ and $y$ if and only if $\mu_{x,y}=1$.) From this graph, we find
that the basis vector with $I(x)=\{1,2,3,5,6\}$ spans the one-dimensional 
intersection of kernels considered above (in accordance with 
\cite[Table~C.4]{ourbuch}). This basis vector labels the last row and column
in the matrix of $Q^{10_s}$ in Table~\ref{wgraphf}.

In this case, the determination of the bound required by James' conjecture
is very easy. By Table~\ref{canf4}, we have $10_s \in
\Lambda_{\zeta_4}^\circ$. If we specialise $v \mapsto \zeta_8$, we notice 
that $Q_{\zeta_4}^{10_s}$ has rank $1$; all rows become equal to 
\[ \left[\begin{array}{c} -2+2\zeta_4, -2+2\zeta_4, 
-2+2\zeta_4, -2\zeta_8^3, -2+2\zeta_4, -2\zeta_8^3, -2+2\zeta_4, 
-2\zeta_8^3, -2\zeta_8^3, -2\zeta_8^3 \end{array}\right] \]
We see that, if we specialise further into a field of characteristic 
$\ell>0$, we will still obtain a matrix of rank $1$ unless $\ell=2$.
\end{exmp}

\medskip
\begin{rem} \label{nono} We have been able to systematically compute the 
matrices $Q^\lambda$ (with coefficients in $A$) for all $\lambda$ such that 
$d_\lambda \leq 2500$. For those $\lambda$ in type $E_8$ where this wasn't 
feasible (at least not with the computer power available to us), we 
nevertheless managed to compute directly the specialized matrices 
$Q^\lambda_{\zeta_e}$ for all relevant values of $e$. Note that this is sufficient 
to find the finite set of prime numbers $\mathcal{P}_e$ as above. (See 
\cite{gemu2} for details.) There is an on-going project to complete the 
determination of all ``generic'' matrices $Q^\lambda$ and to create a data base 
for making them generally available.
\end{rem}


\end{document}